\documentclass{amsart}
\usepackage[margin=1in]{geometry}
\usepackage[all]{xy}
\usepackage{verbatim}
\usepackage{color}
\usepackage{amsthm}
\usepackage{amssymb}
\usepackage[colorlinks=true]{hyperref}
\usepackage{footmisc}
\usepackage{mathtools}

\numberwithin{equation}{section}

\newtheorem{theorem}[equation]{Theorem}
\newtheorem*{theorem*}{Theorem} 
\newtheorem{lemma}[equation]{Lemma}

\newtheorem*{conjecture*}{Mamma Conjecture}
\newtheorem*{conjecture1*}{Mamma Conjecture (revisited)}
\newtheorem{proposition}[equation]{Proposition}
\newtheorem{corollary}[equation]{Corollary}
\newtheorem*{corollary*}{Corollary}

\theoremstyle{remark}

\newtheorem{example}[equation]{Example}

\newtheorem{notation}[equation]{Notation}

\theoremstyle{remark}
\newtheorem{remark}[equation]{Remark}

\newcommand{\cA}{{\mathcal A}}
\newcommand{\cB}{{\mathcal B}}
\newcommand{\cC}{{\mathcal C}}
\newcommand{\cD}{{\mathcal D}}

\newcommand{\cF}{{\mathcal F}}
\newcommand{\cG}{{\mathcal G}}

\newcommand{\cN}{{\mathcal N}}
\newcommand{\cO}{{\mathcal O}}

\newcommand{\bbA}{\mathbb{A}}

\newcommand{\bbC}{\mathbb{C}}

\newcommand{\bbF}{\mathbb{F}}
\newcommand{\bbG}{\mathbb{G}}

\newcommand{\bbL}{\mathbb{L}}
\newcommand{\bbN}{\mathbb{N}}
\newcommand{\bbP}{\mathbb{P}}
\newcommand{\bbR}{\mathbb{R}}

\newcommand{\bbQ}{\mathbb{Q}}
\newcommand{\bbZ}{\mathbb{Z}}

\DeclareMathOperator{\id}{id}

\DeclareMathOperator{\NChow}{NChow} 
\DeclareMathOperator{\NNum}{NNum} 

\newcommand{\dgcat}{\mathrm{dgcat}} 

\newcommand{\perf}{\mathrm{perf}}

\newcommand{\dg}{\mathrm{dg}}

\newcommand{\Hom}{\mathrm{Hom}}

\newcommand{\rep}{\mathrm{rep}}

\newcommand{\Hmo}{\mathrm{Hmo}}
\newcommand{\op}{\mathrm{op}}

\newcommand{\too}{\longrightarrow}

\let\oldmarginpar\marginpar
\def\marginpar#1{\oldmarginpar{\tiny #1}}

\begin{document}

\title[Jacques Tits motivic measure]{Jacques Tits motivic measure}
\author{Gon{\c c}alo~Tabuada}
\address{Gon{\c c}alo Tabuada, Mathematics Institute, Zeeman Building, University of Warwick, Coventry CV4 7AL UK.}
\email{goncalo.tabuada@warwick.ac.uk}
\thanks{The author was partially supported by the Huawei-IH\'ES research funds.}

\date{\today}

\abstract{In this article we construct a new motivic measure called the {\em Jacques Tits motivic measure}. As a first main application of the Jacques Tits motivic measure, we prove that two Severi-Brauer varieties (or, more generally, two twisted Grassmannian varieties), associated to $2$-torsion central simple algebras, have the same class in the Grothendieck ring of varieties if and only if they are isomorphic. In addition, we prove that if two Severi-Brauer varieties, associated to central simple algebras of period $\{3, 4, 5, 6\}$,  have the same class in the Grothendieck ring of varieties, then they are necessarily birational to each other. As a second main application of the Jacques Tits motivic measure, we prove that two quadric hypersurfaces (or, more generally, two involution varieties), associated to quadratic forms of dimension $6$ or to quadratic forms of arbitrary dimension defined over a base field $k$ with $I^3(k)=0$, have the same class in the Grothendieck ring of varieties if and only if they are isomorphic. In addition, we prove that the latter main application  also holds for products of quadric hypersurfaces.}
}

\maketitle


\section{Introduction}\label{sec:intro}
Let $k$ be a field and $\mathrm{Var}(k)$ the category of {\em varieties}, i.e., reduced separated $k$-schemes of finite type. The {\em Grothendieck ring of varieties $K_0\mathrm{Var}(k)$}, introduced in a letter from Grothendieck to Serre (consult \cite[letter of 16/08/1964]{SG}), is defined as the quotient of the free abelian group on the set of isomorphism classes of varieties $[X]$ by the ``cut-and-paste'' relations $[X]=[Y]+[X\backslash Y]$, where $Y$ is a closed subvariety of $X$. The multiplication law is induced by the product of varieties. Despite the efforts of several mathematicians (consult, for example, the works of Bittner \cite{Bittner} and Larsen-Lunts \cite{LL}), the structure of the Grothendieck ring of varieties still remains nowadays poorly understood. In order to capture some of its flavor, a few {\em motivic measures}, i.e., ring homomorphisms $\mu\colon K_0\mathrm{Var}(k) \to R$, have been built. For example, when $k$ is finite the assignment $X \mapsto \# X(k)$ gives rise to the counting motivic measure $\mu_\#\colon K_0\mathrm{Var}(k) \to \bbZ$, and when $k=\bbC$ the assignment $X \mapsto \chi(X):= \Sigma_n (-1)^n \mathrm{dim}_\bbQ\,H^n_c(X^{\mathrm{an}};\bbQ)$ gives rise to the Euler characteristic motivic measure $\mu_\chi\colon K_0\mathrm{Var}(k) \to \bbZ$. In this article we construct a new motivic measure $\mu_{\mathrm{JT}}$ called the {\em Jacques Tits motivic measure}. Making use of it, we then establish several new properties of the Grothendieck ring of varieties.
%
\subsection*{Statement of results}
Let $k$ be a field of characteristic zero and $\Gamma:=\mathrm{Gal}(\overline{k}/k)$ its absolute Galois group.

Recall that given a split semi-simple algebraic group $G$ over $k$, a parabolic subgroup $P \subset G$, and a $1$-cocycle $\gamma\colon \Gamma \to G(\overline{k})$, we can consider the projective homogeneous variety $\cF:=G/P$ as well as its twisted form ${}_\gamma \cF$. Let us write $\widetilde{G}$ and $\widetilde{P}$ for the universal covers of $G$ and $P$, respectively, $R(\widetilde{G})$ and $R(\widetilde{P})$ for the associated representation rings, $n(\cF)$ for the index $[W(\widetilde{G}):W(\widetilde{P})]$ of the Weyl groups, $\widetilde{Z}$ for the center of $\widetilde{G}$, and $Ch$ for the character group $\mathrm{Hom}(\widetilde{Z},\bbG_m)$. As proved by Steinberg in \cite{Steinberg} (consult also \cite[\S12.5-\S12.8]{Panin}), we have $R(\widetilde{P})=\oplus_i R(\widetilde{G})\rho_i$, where $\{\rho_i\}_i$ is a canonical $Ch$-homogeneous basis of cardinality $n(\cF)$. Let us denote by $A_{\rho_i}$ the Jacques Tits central simple $k$-algebra associated to $\rho_i$; consult \cite[\S27]{Book-inv}\cite{Tits} for details.
\begin{notation}
Let us write $K_0\mathrm{Var}(k)^{\mathrm{tw}}$ for the smallest subring of $K_0\mathrm{Var}(k)$ containing the Grothendieck classes $[{}_\gamma \cF]$ of all twisted projective homogeneous varieties ${}_\gamma \cF$.
\end{notation}
Consider the Brauer group $\mathrm{Br}(k)$ of $k$, the associated group ring $\bbZ[\mathrm{Br}(k)]$, and the following quotient ring
$$ R_{\mathrm{B}}(k):= \bbZ[\mathrm{Br}(k)]/ \langle [k]+[A\otimes A'] -[A] - [A'] \,|\,(\mathrm{ind}(A),\mathrm{ind}(A'))=1 \rangle\,,$$
where $A$ and $A'$ are central simple $k$-algebras with coprime indexes. Note that in the particular case where every element of $\mathrm{Br}(k)$ is of $q$-primary torsion for some prime number $q$, it follows from the $p$-primary decomposition of the Brauer group $\mathrm{Br}(k)=\oplus_p \mathrm{Br}(k)\{p\}$ that $R_{\mathrm{B}}(k)$ reduces to the group ring $\bbZ[\mathrm{Br}(k)\{q\}]$. 

\medbreak

Our main result is the following:
\begin{theorem}\label{thm:main}
The assignment ${}_\gamma \cF \mapsto \Sigma_i[A_{\rho_i}]$ gives rise to a motivic measure $\mu_{\mathrm{JT}}\colon K_0\mathrm{Var}(k)^{\mathrm{tw}}\to~R_{\mathrm{B}}(k)$.
\end{theorem}
Intuitively speaking, Theorem \ref{thm:main} shows that the Jacques Tits central simple algebras associated to a twisted projective homogeneous variety are preserved by the ``cut-and-paste'' relations. Motivated by this fact, we decided to call $\mu_{\mathrm{JT}}$ the {\em Jacques Tits motivic measure}. The proof of Theorem \ref{thm:main} makes use, among other ingredients, of the recent theory of noncommutative motives; consult \S\ref{sec:preliminaries}-\S\ref{sec:proof} below.  

\section{Applications}
In this section we describe several applications of the Jacques Tits motivic measure to Severi-Brauer varieties, twisted Grassmannian varieties, quadric hypersurfaces, and involution varieties.

\begin{notation}
Given a central simple $k$-algebra $A$, let us write $\mathrm{deg}(A)$ for its degree, $\mathrm{ind}(A)$ for its index, $\mathrm{per}(A)$ for its period, $[A]$ for its class in the Brauer group $\mathrm{Br}(k)$, and finally $\langle [A]\rangle$ for the subgroup of $\mathrm{Br}(k)$ generated by $[A]$.
\end{notation}
\subsection{Severi-Brauer varieties}\label{sub:Severi}
Let $G$ be the projective general linear group $PGL_n$, with $n\geq 2$. In this case, we have $\widetilde{G}= SL_n$. Consider the following parabolic subgroup:
\begin{eqnarray*}
\widetilde{P}:=\big \{ \begin{pmatrix} a & b \\ 0 & c \end{pmatrix}\,|\, a \cdot \mathrm{det}(c)=1\big \} \subset SL_n & a \in k^\times & c \in GL_{n-1}\,.
\end{eqnarray*}
The associated projective homogeneous variety $\cF:=G/P\simeq \widetilde{G}/\widetilde{P}$ is the projective space $\bbP^{n-1}$ and we have $R(\widetilde{P})=\oplus_{i=0}^{n-1}R(\widetilde{G})\rho_i$. Given a $1$-cocycle $\gamma\colon \Gamma \to PGL_n(\overline{k})$, let $A$ be the corresponding central simple $k$-algebra of degree $n$. Under these notations, the twisted form ${}_\gamma \bbP^{n-1}$ is the Severi-Brauer variety $\mathrm{SB}(A)$ and the Jacques Tits central simple $k$-algebra $A_{\rho_i}$ is the tensor product $A^{\otimes i}$. 
\begin{theorem}\label{thm:app1}
Let $A$ and $A'$ be two central simple $k$-algebras. If $[\mathrm{SB}(A)]=[\mathrm{SB}(A')]$ in the Grothendieck ring of varieties $K_0\mathrm{Var}(k)$, then the following holds:
\begin{itemize}
\item[(i)] We have $\mathrm{dim}(\mathrm{SB}(A))=\mathrm{dim}(\mathrm{SB}(A'))$. Equivalently, we have $\mathrm{deg}(A)=\mathrm{deg}(A')$.
\item[(ii)] We have $\langle [A] \rangle =\langle [A']\rangle$. In particular, we have $\mathrm{per}(A)=\mathrm{per}(A')$.
\item[(iii)] When $[A]\in {}_2\mathrm{Br}(k)$, where ${}_2\mathrm{Br}(k)$ is the $2$-torsion subgroup of $\mathrm{Br}(k)$, we have $\mathrm{SB}(A)\simeq \mathrm{SB}(A')$.
\item[(iv)] When $\mathrm{per}(A)\in \{3, 4, 5, 6\}$, the Severi-Brauer varieties $\mathrm{SB}(A)$ and $\mathrm{SB}(A')$ are birational to each other.
\end{itemize}
\end{theorem}
Note that item (i), resp. item (ii), shows that the dimension of a Severi-Brauer variety, resp. the subgroup generated by the Brauer class, is preserved by the ``cut-and-paste'' relations. Item (iii) shows that when the Brauer class is $2$-torsion (i.e., when $\mathrm{per}(A)\in \{1,2\}$), two Severi-Brauer varieties have the same Grothendieck class in $K_0\mathrm{Var}(k)$ if and only if they are isomorphic! In other words, item (iii) yields the following inclusion:
\begin{equation}\label{eq:injection}
\frac{\{\text{Severi-Brauer}\,\,\text{varieties}\,\,\mathrm{SB}(A)\,\,\text{with}\,\,[A]\in{}_2\mathrm{Br}(k)\}}{\text{isomorphism}} \subset K_0\mathrm{Var}(k) \,.
\end{equation}
Note that thanks to the Artin-Wedderburn theorem, the left-hand side of \eqref{eq:injection} is in bijection with ${}_2\mathrm{Br}(k)\times \bbN$ via the assignment $\mathrm{SB}(A) \mapsto ([A],\mathrm{deg}(A))$. Note also that by restricting the inclusion \eqref{eq:injection} to central simple $k$-algebras of degree $2$, i.e., to quaternion algebras $Q=(a,b)$, we obtain the following inclusion
\begin{equation}\label{eq:injection1}
\frac{\{C(a,b):=(ax^2+ by^2-z^2=0) \subset \bbP^2\,|\,a,b \in k^\times\}}{\text{isomorphism}} \subset K_0\mathrm{Var}(k) \,,
\end{equation}
where $C(a,b)$ stands for the smooth conic associated to the quaternion algebra $Q$.
\begin{example}[Conics over $\bbQ$]\label{ex:coefficients-1}
When $k=\bbQ$, there are infinitely many smooth conics in $\bbP^2$ up to isomorphism. For example, given any two primes numbers $p\neq q$ which are congruent to $3$ modulo $4$, the conics $C(-1,p)$ and $C(-1,q)$ are {\em not} isomorphic. Consequently, since there are infinitely many prime numbers $p$ which are congruent to $3$ modulo $4$, the inclusion \eqref{eq:injection1} yields the following {\em infinite} family of {\em distinct} Grothendieck classes $\{[(-x^2+py^2 - z^2 =0)]\}_{p \equiv 3\,(\text{mod}\,4)} \subset K_0\mathrm{Var}(\bbQ)$.
\end{example}
Finally, item (iv) shows that, for small values of the period, if two Severi-Brauer varieties have the same Grothendieck class in $K_0\mathrm{Var}(k)$ then they are necessarily birational to each other. 
\begin{remark}[Severi-Brauer surfaces]
Let $A$ and $A'$ be two central simple $k$-algebras of degree $3$ (and hence of period $3$). In this particular case, Hogadi proved in \cite[Thm.~1.2]{Hogadi}, using different arguments, that if $[\mathrm{SB}(A)]=[\mathrm{SB}(A')]$ in the Grothendieck ring of varieties $K_0\mathrm{Var}(k)$, then the Severi-Brauer varieties $\mathrm{SB}(A)$ and $\mathrm{SB}(A')$ are birational to each other. Note that, in contrast with Hogadi, in item (iv) we do {\em not} impose any restriction on the degree (only on the period).
\end{remark}
\begin{remark}[Amitsur conjecture]\label{rk:Amitsur}
Let $A$ and $A'$ be two central simple $k$-algebras with the same degree. In the fifties, Amitsur \cite{Amitsur} conjectured that if $\langle [A]\rangle = \langle [A']\rangle$, then the Severi-Brauer varieties $\mathrm{SB}(A)$ and $\mathrm{SB}(A')$ are birational to each other. Consequently, if $[\mathrm{SB}(A)]=[\mathrm{SB}(A')]$ in the Grothendieck ring of varieties $K_0\mathrm{Var}(k)$, we obtain from item (ii) the following variant of item (iv):

\smallskip

\begin{itemize}
\item[(iv')] {\it When the Amitsur conjecture holds, the varieties $\mathrm{SB}(A)$ and $\mathrm{SB}(A')$ are birational to each other.}
\end{itemize}

\smallskip

The Amitsur conjecture holds, for example, when $k$ is a local or global field or when $\mathrm{ind}(A)< \mathrm{deg}(A)$. 
\end{remark}
\begin{remark}[Stable birationality]
Let $A$ and $A'$ be two central simple $k$-algebras with the same degree. It is well-known that if $\langle [A]\rangle = \langle [A']\rangle$, then the Severi-Brauer varieties $\mathrm{SB}(A)$ and $\mathrm{SB}(A')$ are stably birational to each other; consult, for example, \cite[Rk.~5.4.3]{Gille}. Consequently, if $[\mathrm{SB}(A)]=[\mathrm{SB}(A')]$ in the Grothendieck ring of varieties $K_0\mathrm{Var}(k)$, we obtain from item (ii) the following variant of item (iv):

\smallskip

\begin{itemize}
\item[(iv'')] {\it The Severi-Brauer varieties $\mathrm{SB}(A)$ and $\mathrm{SB}(A')$ are stably birational to each other.}
\end{itemize}
\end{remark}
\subsection{Products of conics}
Recall that two quaternion $k$-algebras $Q$ and $Q'$ are called {\em unlinked} in the sense of Albert \cite{Albert} if their tensor product $Q \otimes Q'$ is a division $k$-algebra.
\begin{proposition}\label{prop:app2}
Let $Q=(a,b)$, $Q'=(a',b')$, $Q''=(a'',b'')$, $Q'''=(a''',b''')$ be four quaternion algebras. If $[C(a,b) \times C(a',b')]=[C(a'',b'')\times C(a''',b''')]$ in $K_0\mathrm{Var}(k)$, then the following holds:
\begin{itemize}
\item[(i)] The conic $C(a,b)$ (or $C(a',b')$) is isomorphic to $C(a'',b'')$ or to $C(a''',b''')$.
\item[(ii)] When $Q$ and $Q'$ are unlinked, we have $C(a,b) \times C(a',b')\simeq C(a'',b'') \times C(a''',b''')$.
\end{itemize}
\end{proposition}
Note that item (i) shows that if two products of conics have the same Grothendieck class in $K_0\mathrm{Var}(k)$, then they necessarily share a common conic! Moreover, item (ii) provides a sufficient condition for these two products of conics to be isomorphic.
\begin{example}[Unlinked quaternion algebras]\label{ex:two-variables}
When $k=\bbR(x,y)$ is the field of rational functions on two variables over $\bbR$, the quaternion algebras $(-1,-1)$ and $(x,y)$, as well as the quaternion algebras $(x,-1)$ and $(-x,y)$, are unlinked; consult \cite[\S VI Examples 1.11 and 1.13]{Lam}. In the same vein, when $k= \bbQ(x,y)$ is the field of rational functions on two variables over $\bbQ$, the quaternion algebras $(a,x)$ and $(b,y)$, where $a, b \in k^\times$ represent two independent square classes in $\bbQ^\times/(\bbQ^\times)^2$, are unlinked; consult \cite[\S VI Example 1.15]{Lam}. Further examples exist for every field $k$ with $u$-invariant equal to $6$ or $>8$; consult \cite[\S XIII]{Lam}.
\end{example}
\begin{remark}[Birationality]
Let $k$ be a number field or the function field of an algebraic surface over $\bbC$. Given quaternion algebras $Q=(a,b)$, $Q'=(a',b')$, $Q''=(a'',b'')$, and $Q'''=(a''',b''')$, Koll\'ar\footnote{In subsequent work, Hogadi \cite{Hogadi} removed these restrictions on the base field $k$.} proved in \cite[Thm.~2]{Kollar} that if $[C(a,b) \times C(a',b')]=[C(a'',b'')\times C(a''',b''')]$ in the Grothendieck group of varieties $K_0\mathrm{Var}(k)$, then the products $C(a,b) \times C(a',b')$ and $C(a'',b'') \times C(a''',b''')$ are birational to each other. Hence, the above Proposition \ref{prop:app2} may be understood as a refinement of Kollar's result.
\end{remark}

\subsection{Twisted Grassmannian varieties}\label{sub:Grass}
Let $G=PGL_n$, with $n\geq 2$. Recall that in this case we have $\widetilde{G}= SL_n$. Choose an integer $1\leq d < n$ and consider the following parabolic subgroup:
\begin{eqnarray*}
\widetilde{P}:=\big \{ \begin{pmatrix} a & b \\ 0 & c \end{pmatrix}\,|\, \mathrm{det}(a) \cdot \mathrm{det}(c)=1 \big\} \subset SL_n & a \in GL_d & c \in GL_{n-d}\,.
\end{eqnarray*}
The associated projective homogeneous variety $\cF:=G/P\simeq \widetilde{G}/\widetilde{P}$ is the Grassmannian variety $\mathrm{Gr}(d)$ and we have $R(\widetilde{P})=\oplus_{i}R(\widetilde{G})\rho_i$ where $i=(i_1, \ldots, i_d)$ is a Young diagram inside the rectangle with $d$ lines and $n-d$ columns. Given a $1$-cocycle $\gamma\colon \Gamma \to PGL_n(\overline{k})$, let $A$ be the corresponding central simple $k$-algebra of degree $n$. Under these notations, the twisted form ${}_\gamma \mathrm{Gr}(d)$ is the twisted Grassmannian variety $\mathrm{Gr}(d; A)$ and the Jacques Tits central simple $k$-algebra $A_{\rho_i}$ is the tensor product $A^{\otimes (i_1+ \cdots + i_d)}$.
\begin{remark}[Generalization]\label{rk:general}
Note that in the particular case where $d=1$, the twisted Grassmannian variety $\mathrm{Gr}(d;A)$ reduces to the Severi-Brauer variety $\mathrm{SB}(A)$.
\end{remark}
\begin{theorem}\label{thm:app3}
Let $A$ and $A'$ be two central simple $k$-algebras and $1\leq d < \mathrm{deg}(A)$ and $1\leq d' < \mathrm{deg}(A')$. If $[\mathrm{Gr}(d;A)]=[\mathrm{Gr}(d';A')]$ in the Grothendieck ring of varieties $K_0\mathrm{Var}(k)$, then the following holds:
\begin{itemize}
\item[(i)] We have $\mathrm{dim}(\mathrm{Gr}(d;A))=\mathrm{dim}(\mathrm{Gr}(d';A'))$. Moreover, we have $\mathrm{deg}(A)=\mathrm{deg}(A')$.
\item[(ii)] We have $\langle [A] \rangle =\langle [A']\rangle$. In particular, we have $\mathrm{per}(A)=\mathrm{per}(A')$.
\item[(iii)] When $[A]\in {}_2\mathrm{Br}(k)$, we have $\mathrm{Gr}(d;A)\simeq \mathrm{Gr}(d';A')$.
\end{itemize}
\end{theorem}
Note that, similarly to Theorem \ref{thm:app1}, item (iii) shows that when the Brauer class is $2$-torsion, two twisted Grassmannian varieties have the same Grothendieck class in $K_0\mathrm{Var}(k)$ if and only if they are isomorphic! In other words, item (iii) yields the following inclusion:
\begin{equation}\label{eq:injection3}
\frac{\{\text{Twisted}\,\,\text{Grassmannian}\,\,\text{varieties}\,\,\mathrm{Gr}(d;A)\,\,\text{with}\,\,[A]\in {}_2\mathrm{Br}(k)\}}{\text{isomorphism}} \subset K_0\mathrm{Var}(k)\,.
\end{equation}
Following Remark \ref{rk:general}, note that \eqref{eq:injection3} extends the above inclusion \eqref{eq:injection}.
\subsection{Quadric hypersurfaces}\label{sub:quadrics}
Assume that $\mathrm{char}(k)\neq 2$. Let $G$ be the special orthogonal group $SO_n$, with $n \geq 3$, with respect to the hyperbolic form $\frac{n}{2}\mathbb{H}$ when $n$ is even or to the form $\lfloor{\frac{n}{2}\rfloor}\mathbb{H}\! \perp \!\langle 1\rangle$ when $n$ is odd. In this case, we have $\widetilde{G}=\mathrm{Spin}_n$. Consider the action of $G$ on $\bbP^{n-1}$ given by projective linear transformations, the stabilizer $P\subset G$ of the isotropic point $[1:0:\cdots:0]$, and the pre-image $\widetilde{P}\subset \widetilde{G}$ of $P$. The associated projective homogeneous variety $\cF:=G/P\simeq \widetilde{G}/\widetilde{P}$ is the following smooth quadric hypersurface
\begin{eqnarray*}
Q:=\begin{cases} 
 (x_1y_1 + \cdots + x_{\frac{n}{2}}y_{\frac{n}{2}}=0) \subset \bbP^{n-1} & \quad n\,\,\mathrm{even} \\
 (x_1y_1 + \cdots + x_{\lfloor\frac{n}{2}\rfloor}y_{\lfloor\frac{n}{2}\rfloor} +z^2=0)  \subset \bbP^{n-1} & \quad n\,\,\mathrm{odd} \\
\end{cases}
\end{eqnarray*}
and we have $R(\widetilde{P})=\oplus_{i=0}^{n-1}R(\widetilde{G})\rho_i$ when $n$ is even or  $R(\widetilde{P})=\oplus_{i=0}^{n-2}R(\widetilde{G})\rho_i$ when $n$ is odd. Given a $1$-cocycle $\gamma\colon \Gamma \to SO_n(\overline{k})$, let $q$ be the corresponding non-degenerate quadratic form with trivial discriminant of dimension $n$. Under these notations, the twisted form ${}_\gamma Q$ is the smooth quadric hypersurface $Q_q \subset \bbP^{n-1}$ and the Jacques Tits central simple $k$-algebra $A_{\rho_i}$ is given as follows
\begin{eqnarray*}
A_{\rho_i}:=\begin{cases} k & 0 \leq i \leq n-3 \quad n\,\,\mathrm{even}\\
 C^+_0(q)& i=n-2 \quad n\,\,\mathrm{even}\\
C^-_0(q) & i=n-1 \quad n\,\,\mathrm{even}
\end{cases}
&&
A_{\rho_i}:=\begin{cases} 
k & 0 \leq i \leq n-3 \quad n\,\,\mathrm{odd}\\
C_0(q) & i=n-2 \quad n\,\,\mathrm{odd}\,,
\end{cases}  
\end{eqnarray*}
where $C_0(q)$ stands for the even Clifford algebra of $q$ and $C_0^+(q)$ and $C_0^-(q)$ for the (isomorphic) simple components of $C_0(q)$.
\begin{theorem}\label{thm:app4}
Let $q$ and $q'$ be two non-degenerate quadratic forms with trivial discriminant of dimensions $n$ and $n'$, respectively. If $[Q_q]=[Q_{q'}]$ in the Grothendieck ring of varieties $K_0\mathrm{Var}(k)$, then the following~holds:
\begin{itemize}
\item[(i)] We have $\mathrm{dim}(Q_q)=\mathrm{dim}(Q_{q'})$. Equivalently, we have $n=n'$.
\item[(ii)] We have $C^+_0(q)\simeq C^+_0(q')$ when $n$ is even or $C_0(q)\simeq C_0(q')$ when $n$ is odd.
\item[(iii)] When $n=6$, we have $Q_q\simeq Q_{q'}$.
\item[(iv)] When $I^3(k)=0$, where $I(k)\subset W(k)$ is the fundamental ideal of the Witt ring, we have $Q_q\simeq Q_{q'}$.
\end{itemize}
\end{theorem}
Note that item (i), resp. item (ii), shows that the dimension of the quadric hypersurface, resp. the Brauer class of the (simple components of the) even Clifford algebra, is preserved by the ``cut-and-paste'' relations. Item (iii) shows that when the dimension is equal to $6$, two quadric hypersurfaces have the same Grothendieck class in $K_0\mathrm{Var}(k)$ if and only if they are isomorphic! Recall from \cite[\S16.4]{Book-inv} that, up to similarity, a non-degenerate quadratic form $q$ with trivial discriminant of dimension $6$ is given by $\langle a, b, -ab, -a', -b', a'b'\rangle$ with $a, b, a', b' \in k^\times$. Therefore, item (iii) yields the following inclusion:
\begin{equation}\label{eq:injection-dim4}
\frac{\{Q_q:=(au^2 + bv^2 -abw^2 - a'x^2 -b'y^2 + a'b'z^2=0)\subset \bbP^5\,|\,a, b, a', b' \in k^\times\}}{\text{isomorphism}} \subset K_0\mathrm{Var}(k) \,.
\end{equation}
\begin{example}[Quadric hypersurfaces over $\bbQ$]\label{ex:coefficients-2}
When $k=\bbQ$, there are infinitely many quadric hypersurfaces in $\bbP^5$ up to isomorphism. For example, we have the following {\em infinite} family of {\em non-isomorphic} quadric hypersurfaces $\{(u^2 + v^2 - w^2 + x^2 -py^2 -pz^2=0)\}_{p\equiv 3 \,(\text{mod}\,4)}$ parametrized by the prime numbers $p$ which are congruent to $3$ modulo $4$. Making use of \eqref{eq:injection-dim4}, we hence obtain the following {\em infinite} family of {\em distinct} Grothendieck classes $\{[(u^2 + v^2 - w^2 + x^2 -py^2 -pz^2=0)]\}_{p\equiv 3 \,(\text{mod}\,4)} \subset K_0\mathrm{Var}(\bbQ)$. Note that since all these quadric hypersurfaces have a rational $k$-point, it follows from \cite[Thm.~1.11]{Kollar-book} that they are all birational to $\bbP^4$. This shows that, in the case of quadric hypersurfaces, the Grothendieck class in $K_0\mathrm{Var}(k)$ contains much more information than the birational equivalence class.
\end{example}
Finally, item (iv) shows that when $I^3(k)=0$, two quadric hypersurfaces have the same Grothendieck class in $K_0\mathrm{Var}(k)$ if and only if they are isomorphic! Consequently, in this case, the above inclusion \eqref{eq:injection-dim4} admits the following far-reaching extension:
\begin{eqnarray}\label{eq:inclusion-verylast}
\frac{\{\text{Quadric}\,\,\text{hypersurfaces}\,\,Q_q\,\,\text{with}\,\,\text{trivial}\,\,\text{discriminant}\}}{\text{isomorphism}} \subset K_0\mathrm{Var}(k) && I^3(k)=0\,.
\end{eqnarray}
Recall that $I^3(k)=0$ when $k$ is a $C_2$-field or, more generally, when $k$ is a field of cohomological dimension $\leq 2$. Examples include fields of transcendence degree $\leq 2$ over algebraically closed fields, $p$-adic fields, non formally real global fields, etc.
\subsection{Products of quadrics}\label{sub:quadrics-products}
Surprisingly, Theorem \ref{thm:app4} admits the following generalization:
\begin{theorem}\label{thm:prod-quadrics}
Let $\{q_j\}_{1\leq j \leq m}$ and $\{q'_j\}_{1\leq j \leq m'}$ be two families of non-degenerate quadratic forms with trivial discriminant of dimension $n\geq 5$. If $[\Pi_{j=1}^{m} Q_{q_j}]=[\Pi_{j=1}^{m'} Q_{q'_j}]$ in $K_0\mathrm{Var}(k)$, then the following holds:
\begin{itemize}
\item[(i)] We have $\mathrm{dim}(\Pi_{j=1}^m Q_{q_j})=\mathrm{dim}(\Pi_{j=1}^{m'}Q_{q'_j})$. Equivalently, we have $m=m'$.
\item[(ii)] We have $\langle \{[C^+_0(q_j)]\}_j \rangle = \langle \{[C^+_0(q'_j)]\}_j \rangle$ when $n$ is even or $\langle \{[C_0(q_j)]\}_j \rangle=\langle \{[C_0(q'_j)]\}_j \rangle$ when $n$ is odd.
\item[(iii)] When $n=6$ and $m\leq 5$, we have $\Pi_{j=1}^m Q_{q_j}\simeq \Pi_{j=1}^{m}Q_{q'_j}$.
\item[(iv)] When $I^3(k)=0$ and $m \leq 5$, we have $\Pi_{j=1}^m Q_{q_j}\simeq \Pi_{j=1}^{m}Q_{q'_j}$.
\item[(iv')] When $I^3(k)=0$, $m\geq 6$, and the following extra condition holds (consult Notation \ref{not:sums} below)
\begin{equation}\label{eq:extra-sums}
\begin{cases} \Sigma^1_{\mathrm{even}}(m,n,l)> \Sigma^2_{\mathrm{even}}(m,n,l)\,\,\,\mathrm{for}\,\,\,\mathrm{every}\,\,\,2 \leq l \leq m-3 & \quad \mathrm{n}\,\,\mathrm{even} \\
\Sigma^1_{\mathrm{odd}}(m,n,l)> \Sigma^2_{\mathrm{odd}}(m,n,l)\,\,\,\mathrm{for}\,\,\,\mathrm{every}\,\,\,2 \leq l \leq m-3 & \quad \mathrm{n}\,\,\mathrm{odd}\,,
\end{cases}
\end{equation}
we also have $\Pi_{j=1}^m Q_{q_j}\simeq \Pi_{j=1}^{m} Q_{q'_j}$.
\end{itemize}
\end{theorem}
\begin{notation}\label{not:sums}
Given integers $m,n, l\geq 0$, consider the following sums of $(n-2)$-powers:
\begin{eqnarray*}
\Sigma^1_{\mathrm{even}}(m,n,l) & := &\Sigma^{\lfloor l/2 \rfloor}_{r=0} \big(\tbinom{l}{2r} \times 2^{2r+1} \times (n-2)^{m-(2r+1)} + \tbinom{l}{2r+1} \times 2^{m-l+ (2r+1)} \times (n-2)^{l-(2r+1)} \big) \\
\Sigma^1_{\mathrm{odd}}(m,n,l)& :=&\Sigma^{\lfloor l/2 \rfloor}_{r=0} \big(\tbinom{l}{2r} \times (n-2)^{m-(2r+1)} + \tbinom{l}{2r+1} \times (n-2)^{l-(2r+1)} \big) \\
\Sigma^2_{\mathrm{even}}(m,n,l)& :=&\Sigma^{\lfloor l/2 \rfloor}_{r=0} \big(\tbinom{l}{2r} \times 2^{2r+2} \times (n-2)^{m-(2r+2)} + \tbinom{l}{2r+1} \times 2^{2r+1} \times (n-2)^{m-(2r+2)} \big) \\
\Sigma^2_{\mathrm{odd}}(m,n,l)& := &\Sigma^{\lfloor l/2 \rfloor}_{r=0} \big(\tbinom{l}{2r} \times (n-2)^{m-(2r+2)} + \tbinom{l}{2r+1} \times (n-2)^{m-(2r+2)} \big)\,.
\end{eqnarray*} 
\end{notation}
Note that item (iii) shows that when the dimension is equal to $6$, two products of quadrics (with $m\leq 5$) have the same Grothendieck class in $K_0\mathrm{Var}(k)$ if and only if they are isomorphic! This implies that the above inclusion \eqref{eq:injection-dim4} holds more generally for products of quadrics (with $m\leq 5$). In the same vein, items (iv)-(iv') show that when $I^3(k)=0$, two products of quadrics have the same Grothendieck class in $K_0\mathrm{Var}(k)$ if and only if they are isomorphic! This implies that the above inclusion \eqref{eq:inclusion-verylast} holds more generally for products of quadrics. Finally, note that the extra condition \eqref{eq:extra-sums} holds whenever the dimension $n$ is $\gg$ than the number $m$ of quadrics because the highest power of $(n-2)$ in the sums $\Sigma^1_{\mathrm{even}}(m,n,l)$ and $\Sigma^1_{\mathrm{odd}}(m,n,l)$ is $(n-2)^{m-1}$ while the highest power of $(n-2)$ in the sums $\Sigma^2_{\mathrm{even}}(m,n,l)$ and $\Sigma^2_{\mathrm{odd}}(m,n,l)$ is $(n-2)^{m-2}$.

\subsection{Involution varieties}\label{sub:involution}
Assume that $\mathrm{char}(k)\neq 2$. Let $G$ be the projective special orthogonal group $PSO_n$, with $n \geq  6$ even, with respect to the hyperbolic form $\frac{n}{2}\mathbb{H}$. In this case, we have $\widetilde{G}=\mathrm{Spin}_n$. Similarly to \S\ref{sub:quadrics}, consider the projective homogeneous variety $\cF$ given by $Q :=(x_1y_1 + \cdots + x_{\frac{n}{2}}y_{\frac{n}{2}}=0)  \subset \bbP^{n-1}$ and recall from {\em loc. cit.} that $R(\widetilde{P})=\oplus_{i=0}^{n-1}R(\widetilde{G})\rho_i$. Given a $1$-cocycle $\gamma \colon \Gamma \to PSO_n(\overline{k})$, let $(A,\ast)$ be the corresponding central simple $k$-algebra of degree $n$ with involution of orthogonal type and trivial discriminant. Under these notations, the twisted form ${}_\gamma Q$ is the involution variety $\mathrm{Iv}(A,\ast)\subset \bbP^{n-1}$ and the Jacques Tits central simple $k$-algebra $A_{\rho_i}$ is given as follows
$$
A_{\rho_i}:=\begin{cases} k & 0 \leq i \leq n-3 \quad i\,\,\text{even} \\
A & 0 \leq i \leq n-3 \quad i\,\,\text{odd} \\
C^+_0(A,\ast) & i=n-2 \\
C^-_0(A,\ast) & i=n-1\,,
\end{cases}
$$
where $C^+_0(A,\ast)$ and $C^-_0(A,\ast)$ stand for the simple components of the even Clifford algebra $C_0(A,\ast)$ of $(A,\ast)$.
\begin{remark}[Generalization]\label{rk:split}
In the particular case where $(A,\ast)$ is {\em split}, i.e., isomorphic to $(M_n(k), \ast_q)$ with $\ast_q$ the adjoint involution associated to a quadratic form $q$, the involution variety $\mathrm{Iv}(A,\ast)\subset \bbP^{n-1}$ reduces to the quadric hypersurface $Q_q \subset \bbP^{n-1}$. Hence, involution varieties may be understood as ``forms  of quadrics''.
\end{remark}
\begin{theorem}\label{thm:app5}
Let $(A,\ast)$ and $(A',\ast')$ be two central simple $k$-algebras with involutions of orthogonal type and trivial discriminant. If $[\mathrm{Iv}(A,\ast)]=[\mathrm{Iv}(A',\ast')]$ in $K_0\mathrm{Var}(k)$, then the following holds:
\begin{itemize}
\item[(i)] We have $\mathrm{dim}(\mathrm{Iv}(A,\ast))=\mathrm{dim}(\mathrm{Iv}(A',\ast'))$. Equivalently, we have $\mathrm{deg}(A)=\mathrm{deg}(A')$.
\item[(ii)] We have\footnote{The short notation $C_0^\pm(A,\ast)\simeq C_0^\pm(A',\ast')$ stands for $\begin{cases} C^{+}_0(A,\ast)\simeq C^{+}_0(A',\ast') \\ C^{-}_0(A,\ast)\simeq C^{-}_0(A',\ast')\end{cases}$ $\!\!\!\!\!$ or $\,\,$ $\begin{cases}C^{+}_0(A,\ast)\simeq C^{-}_0(A',\ast') \\ C^{-}_0(A,\ast)\simeq C^{+}_0(A',\ast')\,.\end{cases}$} $C_0^\pm(A,\ast)\simeq C_0^\pm(A',\ast')$.
\item[(iii)] When $\mathrm{deg}(A)=6$, we have $\mathrm{Iv}(A,\ast)\simeq \mathrm{Iv}(A',\ast')$.
\item[(iv)] When $I^3(k)=0$, we have $\mathrm{Iv}(A,\ast)\simeq \mathrm{Iv}(A',\ast')$.
\end{itemize}
\end{theorem}
Note that, similarly to Theorem \ref{thm:app4}, item (iii), resp. (iv), shows that when the degree is equal to $6$, resp. $I^3(k)=0$, two involution varieties have the same Grothendieck class in $K_0\mathrm{Var}(k)$ if and only if they are isomorphic! In other words, items (iii)-(iv) yield the following inclusions:
\begin{equation}\label{eq:inclusion-last}
\frac{\{\text{Involution}\,\,\text{varieties}\,\,\mathrm{Iv}(A,\ast)\,\,\text{with}\,\,\text{trivial}\,\,\text{discriminant}\,\,\text{and}\,\,\text{deg}(A)=6\}}{\text{isomorphism}} \subset K_0\mathrm{Var}(k)
\end{equation}
\begin{eqnarray}\label{eq:inclusion-last1}
\frac{\{\text{Involution}\,\,\text{varieties}\,\,\mathrm{Iv}(A,\ast)\,\,\text{with}\,\,\text{trivial}\,\,\text{discriminant}\}}{\text{isomorphism}} \subset K_0\mathrm{Var}(k) && I^3(k)=0\,.
\end{eqnarray}
Following Remark \ref{rk:split}, note that \eqref{eq:inclusion-last}, resp. \eqref{eq:inclusion-last1}, extends the above inclusion \eqref{eq:injection-dim4}, resp. \eqref{eq:inclusion-verylast}.
\section{Preliminaries}\label{sec:preliminaries}
Throughout the article $k$ denotes a base field of characteristic zero.
\begin{notation}
Given a central simple $k$-algebra $A$ and a prime number $q$, let us write $[A]^q\in \mathrm{Br}(k)\{q\}$ for the $q$-primary component of the Brauer class $[A] \in \mathrm{Br}(k)=\oplus_p\mathrm{Br}(k)\{p\}$.
\end{notation}
\subsection*{Dg categories}
A {\em differential graded (=dg) category $\cA$} is a category enriched over complexes of $k$-vector spaces; consult Keller's survey \cite{ICM-Keller} (and Bondal-Kapranov's original article \cite{BK}). Every (dg) $k$-algebra $A$ gives naturally rise to a dg category with a single object. Another source of examples is provided by schemes since the category of perfect complexes $\perf(X)$ of every $k$-scheme $X$ admits a canonical dg enhancement $\perf_\dg(X)$; consult \cite[\S4.6]{ICM-Keller}. Let us denote by $\dgcat(k)$ the category of (small) dg categories.

Let $\cA$ be a dg category. The {\em opposite dg category} $\cA^\op$ has the same objects as $\cA$ and $\cA^\op(x,y):=\cA(y,x)$. A {\em right dg $\cA$-module} is a dg functor $M:\cA^\op \to \cC_\dg(k)$ with values in the dg category of complexes of $k$-vector spaces. Let us denote by $\cC(\cA)$ the category of right dg $\cA$-modules. Following \cite[\S3.2]{ICM-Keller}, the {\em derived category $\cD(\cA)$ of $\cA$} is defined as the localization of $\cC(\cA)$ with respect to the objectwise quasi-isomorphisms. In what follows, we will write $\cD_c(\cA)$ for the subcategory of compact objects. 

A dg functor $F\colon \cA\to \cB$ is called a {\em Morita equivalence} if it induces an equivalence on derived categories $\cD(\cA) \simeq \cD(\cB)$; consult \cite[\S4.6]{ICM-Keller}. As explained in \cite[\S1.6.1]{book}, $\dgcat(k)$ admits a Quillen model structure whose weak equivalences are the Morita equivalences. Let $\Hmo(k)$ be the associated homotopy category.

The {\em tensor product $\cA\otimes\cB$} of dg categories is defined as follows: the set of objects is the cartesian product of the sets of objects and $(\cA\otimes\cB)((x,w),(y,z)):= \cA(x,y) \otimes \cB(w,z)$. As explained in \cite[\S2.3]{ICM-Keller}, this construction gives rise to a symmetric monoidal structure on $\dgcat(k)$ which descends to the homotopy category $\Hmo(k)$. 

A {\em dg $\cA\text{-}\cB$-bimodule} is a dg functor $\mathrm{B}:\cA\otimes \cB^\op \to \cC_\dg(k)$. For example, given a dg functor $F\colon \cA\to \cB$, we have the dg $\cA\text{-}\cB$-bimodule ${}_F\cB\colon \cA\otimes \cB^\op \to \cC_\dg(k), (x,z) \mapsto \cB(z, F(x))$. Let $\rep(\cA,\cB)$ be the full triangulated subcategory of $\cD(\cA^\op \otimes \cB)$ consisting of those dg $\cA\text{-}\cB$-bimodules $\mathrm{B}$ such that for every object $x \in \cA$ the right dg $\cB$-module $\mathrm{B}(x,-)$ belongs to $\cD_c(\cB)$. Clearly, the dg $\cA\text{-}\cB$-bimodules ${}_F\cB$ belongs to $\rep(\cA,\cB)$.

Finally, following Kontsevich \cite{Miami,finMot,IAS}, a dg category $\cA$ is called {\em smooth} if the dg $\cA\text{-}\cA$ bimodule ${}_{\id} \cA$ belongs to $\cD_c(\cA^\op \otimes \cA)$ and {\em proper} if $\Sigma_n \mathrm{dim}\,H^n\cA(x, y) < \infty$ for any pair of objects $(x, y)$. Examples include finite-dimensional $k$-algebras of finite global dimension $A$ as well as the dg categories of perfect complexes $\perf_\dg(X)$ associated to smooth proper $k$-schemes $X$.
\subsection*{Noncommutative motives}
For a book on noncommutative motives, we invite the reader to consult \cite{book}. As explained in \cite[\S1.6.3]{book}, given any two dg categories $\cA$ and $\cB$, there is a natural bijection between $\Hom_{\Hmo(k)}(\cA,\cB)$ and the set of isomorphism classes of the category $\rep(\cA,\cB)$. Under this bijection, the composition in $\Hmo(k)$ corresponds to the (derived) tensor product of bimodules.
The {\em additivization} of $\Hmo(k)$ is the additive category $\Hmo_0(k)$ with the same objects and with abelian groups of morphisms $\Hom_{\Hmo_0(k)}(\cA,\cB)$ given by the Grothendieck group $K_0\rep(\cA,\cB)$ of the triangulated category $\rep(\cA,\cB)$. The composition law is induced by the (derived) tensor product of bimodules. Given a commutative ring of coefficients $R$, the {\em $R$-linearization} of $\Hmo_0(k)$ is the $R$-linear category $\Hmo_0(k)_R$ obtained by tensoring the morphisms of $\Hmo_0(k)$ with $R$. Note that we have the following (composed) symmetric monoidal functor
\begin{eqnarray*}\label{eq:func3}
U(-)_R\colon \dgcat(k) \too \Hmo_0(k)_R & \cA \mapsto \cA & (\cA\stackrel{F}{\to} \cB) \mapsto [{}_F\cB]_R\,.
\end{eqnarray*}

The category of {\em noncommutative Chow motives} $\NChow(k)_R$ is defined as the idempotent completion of the full subcategory of $\Hmo_0(k)_R$ consisting of the objects $U(\cA)_R$ with $\cA$ a smooth proper dg category. This category is $R$-linear, additive, rigid symmetric monoidal, and idempotent complete. When $R=\bbZ$, we will write $\NChow(k)$ instead of $\NChow(k)_\bbZ$ and $U(-)$ instead of $U(-)_\bbZ$.

Given an additive rigid symmetric monoidal category $\cC$, recall that its {\em $\cN$-ideal} is defined as follows
$$ \cN(a,b):=\{f \in \Hom_\cC(a,b)\,|\, \forall g \in \Hom_\cC(b,a)\,\,\mathrm{we}\,\,\mathrm{have}\,\,\mathrm{tr}(g\circ f)=0 \}\,,$$
where $\mathrm{tr}(g\circ f)$ stands for the categorical trace of $g\circ f$. The category of {\em noncommutative numerical motives} $\NNum(k)_R$ is defined as the idempotent completion of the quotient of $\NChow(k)_R$ by the $\otimes$-ideal $\cN$. By construction, this category  is $R$-linear, additive, rigid symmetric monoidal, and idempotent complete. 
\begin{remark}[Smooth proper schemes]\label{rk:generalization}
Given any two smooth projective $k$-schemes $X$ and $Y$, we have the following Morita equivalence (consult \cite[Lem.~4.26]{Gysin}):
\begin{eqnarray*}
\perf_\dg(X) \otimes \perf_\dg(Y) \too \perf_\dg(X\times Y) && (\cF,\cG) \mapsto \cF\boxtimes \cG\,.
\end{eqnarray*}
Using the fact that the functor $U(-)_R$ is symmetric monoidal, we hence conclude that the noncommutative (Chow or numerical) motives $U(\perf_\dg(X\times Y))_R$ and $U(\perf_\dg(X))_R \otimes U(\perf_\dg(Y))_R$ are isomorphic.
\end{remark}
\section{Proof of Theorem \ref{thm:main}}\label{sec:proof}
Let $K_0(\NChow(k))$ be the Grothendieck ring of the additive symmetric monoidal category of noncommutative Chow motives $\NChow(k)$. We start by constructing a motivic measure with values in this ring.
\begin{proposition}\label{prop:2}
The assignment $X \mapsto U(\perf_\dg(X))$, with $X$ a smooth projective $k$-scheme, gives rise to a motivic measure $\mu_{\mathrm{nc}}\colon K_0\mathrm{Var}(k) \to K_0(\NChow(k))$.
\end{proposition}
\begin{proof}
Thanks to Bittner's presentation \cite[Thm.~3.1]{Bittner} of the Grothendieck ring of varieties $K_0\mathrm{Var}(k)$, it suffices to prove the following two conditions:
\begin{itemize}
\item[(i)] Let $X$ be a smooth projective $k$-scheme, $Y \hookrightarrow X$ a smooth closed subscheme of codimension $c$, $\mathrm{Bl}_Y(X)$ the blow-up of $X$ along $Y$, and $E$ the exceptional divisor of the blow-up. Under these notations, we have the following equality in the Grothendieck ring $K_0(\NChow(k))$:
$$ [U(\perf_\dg(\mathrm{Bl}_Y(X)))]-[U(\perf_\dg(E))] = [U(\perf_\dg(X))]- [U(\perf_\dg(Y))]\,.$$
\item[(ii)] Given smooth projective $k$-schemes $X$ and $Y$, we have the following equality in $K_0(\NChow(k))$: 
$$ [U(\perf_\dg(X\times Y))] = [U(\perf_\dg(X))\otimes U(\perf_\dg(Y))]\,.$$
\end{itemize}
Let us write $f\colon \mathrm{Bl}_Y(X) \to X$ for the blow-up map, $i\colon E \hookrightarrow \mathrm{Bl}_Y(X)$ for the embedding map, and $p\colon E \to Y$ for the projection map (= restriction of $f$ to $E$). Under these notations, recall from Orlov \cite[Thm.~4.3]{Orlov} that we have the following semi-orthogonal decompositions
\begin{eqnarray*}
\perf(\mathrm{Bl}_Y(X)) & = & \langle \bbL f^\ast(\perf(X)), \Psi_0(\perf(Y)), \ldots, \Psi_{c-2}(\perf(Y)) \rangle \\
 \perf(E) & = & \langle \Phi_0(\perf(Y)), \ldots, \Phi_{c-1}(\perf(Y))\rangle \,,
\end{eqnarray*} 
where $\Psi_i(-):=\bbR i_\ast(\bbL p^\ast(-) \otimes^{\bbL} \cO_{E/Y}(i))$ and $\Phi_i(-):=\bbL p^\ast(-)\otimes^{\bbL} \cO_{E/Y}(i)$. Moreover, the functors $\bbL f^\ast(-)$, $\Psi_i(-)$, and $\Phi_i(-)$, are fully-faithful and of Fourier-Mukai type. Consequently, using the fact that the functor $U$ sends semi-orthogonal decomposition to direct sums (consult \cite[\S2.1]{book}), we conclude that 
\begin{eqnarray*} 
[U(\perf_\dg(\mathrm{Bl}_Y(X)))]=[U(\perf_\dg(X))]+ (c-1) [U(\perf_\dg(Y))] &\mathrm{and}& [U(\perf_\dg(E))]= c[U(\perf_\dg(Y))]\,.
\end{eqnarray*}
These equalities imply condition (i). Condition (ii) follows now from the above Remark \ref{rk:generalization}.
\end{proof}
\begin{notation}\label{not:CSA}
Given a commutative ring of coefficients $R$, let us write $\mathrm{CSA}(k)_R$ for the full subcategory of $\NChow(k)_R$ consisting of the objects $U(A)_R$ with $A$ a central simple $k$-algebra. The closure of $\mathrm{CSA}(k)_R$ under finite direct sums will be denoted by $\mathrm{CSA}(k)^{\oplus}_R$. In the same vein, let us write $\underline{\mathrm{CSA}}(k)_R$ for the full subcategory of $\NNum(k)_R$ consisting of the objects $U(A)_R$ with $A$ a central simple $k$-algebra, and $\underline{\mathrm{CSA}}(k)^{\oplus}_R$ for its closure under finite direct sums. In the particular case where $R=\bbZ$, we will omit the~underscript~$(-)_\bbZ$. 
\end{notation}
\begin{proposition}\label{prop:p-Brauer}
Given two central simple $k$-algebras $A$ and $A'$ and a prime number $p$, the following holds:
\begin{itemize}
\item[(i)] We have an isomorphism $U(A)_{\bbF_p}\simeq U(A')_{\bbF_p}$ in the category $\mathrm{CSA}(k)_{\bbF_p}$ (or, equivalently, in the category $\underline{\mathrm{CSA}}(k)_{\bbF_p}$) if and only if $[A]^p=[A']^p$ in $\mathrm{Br}(k)\{p\}$.
\item[(ii)] The $\bbF_p$-vector space $\Hom_{\underline{\mathrm{CSA}}(k)_{\bbF_p}}(U(A)_{\bbF_p}, U(A')_{\bbF_p})$ is naturally isomorphic to $\begin{cases} \bbF_p& \text{when}\,\,[A]^p=[A']^p \\ 0 & \text{when}\,\, [A]^p\neq [A']^p\,.\end{cases}$
\end{itemize}
\end{proposition}
\begin{proof}
Given central simple $k$-algebras $A$, $A'$, $A''$, recall from \cite[Prop.~2.25]{Separable} that the composition map
$$
\Hom_{\mathrm{CSA}(k)}(U(A),U(A')) \times \Hom_{\mathrm{CSA}(k)}(U(A'),U(A'')) \too \Hom_{\mathrm{CSA}(k)}(U(A),U(A''))
$$
corresponds to the following bilinear pairing:
\begin{eqnarray*}
\bbZ \times \bbZ \too \bbZ && (n,m) \mapsto n\cdot \mathrm{ind}(A^\op \otimes A') \cdot \mathrm{ind}(A'^\op\otimes A'') \cdot m\,.
\end{eqnarray*}
Since $\mathrm{ind}(A^\op \otimes A')=\mathrm{ind}(A'^\op \otimes A)$, this implies, in particular, that the composition map
$$
\Hom_{\mathrm{CSA}(k)_{\bbF_p}}(U(A)_{\bbF_p},U(A')_{\bbF_p}) \times \Hom_{\mathrm{CSA}(k)_{\bbF_p}}(U(A')_{\bbF_p},U(A)_{\bbF_p}) \too \Hom_{\mathrm{CSA}(k)_{\bbF_p}}(U(A)_{\bbF_p},U(A)_{\bbF_p})
$$
corresponds to the following bilinear pairing
\begin{eqnarray}\label{eq:bilinear-key}
\bbF_p \times \bbF_p \too \bbF_p && (n,m) \mapsto n\cdot \mathrm{ind}(A^\op \otimes A')^2 \cdot m\,;
\end{eqnarray}
similarly with $A$ and $A'$ replaced by $A'$ and $A$, respectively. On the one hand, if $p\mid \mathrm{ind}(A^\op \otimes A')$, the bilinear pairing \eqref{eq:bilinear-key} is equal to zero. This implies that $U(A)_{\bbF_p}\not\simeq U(A')_{\bbF_p}$ in the category $\mathrm{CSA}(k)_{\bbF_p}$ and, by definition of the category $\underline{\mathrm{CSA}}(k)_{\bbF_p}$, that $\Hom_{\underline{\mathrm{CSA}}(k)_{\bbF_p}}(U(A)_{\bbF_p}, U(A')_{\bbF_p})=0$. In particular, we also have $U(A)_{\bbF_p}\not\simeq U(A')_{\bbF_p}$ in the category $\underline{\mathrm{CSA}}(k)_{\bbF_p}$. On the other hand, if $p\nmid \mathrm{ind}(A^\op \otimes A')$, the integer $\mathrm{ind}(A^\op \otimes A')$ is invertible in $\bbF_p$. Consequently, we conclude from \eqref{eq:bilinear-key} that $U(A)_{\bbF_p}\simeq U(A')_{\bbF_p}$ in the category $\mathrm{CSA}(k)_{\bbF_p}$ (and hence in the category $\underline{\mathrm{CSA}}(k)_{\bbF_p}$). Thanks to the definition of the category $\underline{\mathrm{CSA}}(k)_{\bbF_p}$, this implies moreover that $\Hom_{\underline{\mathrm{CSA}}(k)_{\bbF_p}}(U(A)_{\bbF_p}, U(A')_{\bbF_p})\simeq\bbF_p$. Finally, note that the $p$-primary decomposition of the Brauer group $\mathrm{Br}(k)=\oplus_p \mathrm{Br}(k)\{p\}$ implies that $p\nmid \mathrm{ind}(A^\op \otimes A')$ if and only if $[A]^p=[A']^p$ in $\mathrm{Br}(k)\{p\}$.
\end{proof}
Note that Proposition \ref{prop:p-Brauer} implies, in particular, the following result:
\begin{corollary}\label{cor:graded}
Given a prime number $p$, we have an induced equivalence of categories
\begin{eqnarray}\label{eq:equivalence-induced}
\underline{\mathrm{CSA}}(k)^{\oplus}_{\bbF_p} \stackrel{\simeq}{\too} \mathrm{Vect}_{\mathrm{Br}(k)\{p\}}(k) && U(A_1)_{\bbF_p}\oplus \cdots \oplus U(A_m)_{\bbF_p} \mapsto (\bbF_p [A_1]^p) \oplus \cdots \oplus (\bbF_p[A_m]^p)\,,
\end{eqnarray}
where $\mathrm{Vect}_{\mathrm{Br}(k)\{p\}}(k)$ stands for the category of finite-dimensional $\mathrm{Br}(k)\{p\}$-graded $\bbF_p$-vector spaces. 
\end{corollary}
Recall from \cite[Thm.~2.20(iv)]{Separable} the following result:
\begin{proposition}\label{prop:key}
Given two families of central simple $k$-algebras $\{A_j\}_{1\leq j \leq m}$ and $\{A'_j\}_{1\leq j \leq m'}$, the following conditions are equivalent:
\begin{itemize}
\item[(i)] We have an isomorphism $\oplus_{j=1}^m U(A_j)\simeq \oplus_{j=1}^{m'} U(A'_j)$ in the category $\mathrm{CSA}(k)^\oplus$.
\item[(ii)] We have $m=m'$ and for every prime number $p$ there exists a permutation $\sigma_p$ (which depends on $p$) such that $[A'_j]^p=[A_{\sigma_p(j)}]^p$ in $\mathrm{Br}(k)\{p\}$ for every $1 \leq j \leq m$. 
\end{itemize}
\end{proposition}
\begin{corollary}\label{cor:key}
Let $\{A_j\}_{1\leq j \leq m}$ and $\{A'_j\}_{1\leq j \leq m'}$ be two families of central simple $k$-algebras. If we have an isomorphism $\oplus_{j=1}^m U(A_j)\simeq \oplus_{j=1}^{m'} U(A'_j)$ in the category $\mathrm{CSA}(k)^\oplus$, then $\langle\{[A_j]\}_{1\leq j \leq m} \rangle =\langle \{[A'_j]\}_{1\leq j \leq m'}\rangle$.
\end{corollary}
The following result is of independent interest:
\begin{proposition}[Cancellation]\label{prop:cancellation}
Let $\{A_j\}_{1\leq j \leq m}$ and $\{A'_j\}_{1\leq j \leq m'}$ be two families of central simple $k$-algebras and $N\!\!M$ a noncommutative Chow motive. If $\oplus_{j=1}^m U(A_j) \oplus N\!\!M\simeq \oplus^{m'}_{j=1} U(A'_j)\oplus N\!\!M$ in the category $\mathrm{NChow}(k)$, then $m=m'$ and $\oplus^m_{j=1} U(A_j) \simeq \oplus^{m}_{j=1} U(A'_j)$.
\end{proposition}
\begin{proof}
Given a (fixed) prime number $p$, consider the induced isomorphism
\begin{equation}\label{eq:isom-1}
\oplus_{j=1}^m U(A_j)_{\bbF_p} \oplus N\!\!M_{\bbF_p} \simeq \oplus^{m'}_{j=1} U(A'_j)_{\bbF_p} \oplus N\!\!M_{\bbF_p}
\end{equation}
in the category $\NNum(k)_{\bbF_p}$. Thanks to Lemma \ref{lem:splitting} below, there exists a noncommutative numerical motive $M\!\!N$ and integers $r_j, r'_j\geq 0$ such that 
$N\!\!M_{\bbF_p}\simeq \oplus^m_{j=1} U(A_j)_{\bbF_p}^{\oplus r_j} \oplus \oplus^{m'}_{j=1} U(A'_j)_{\bbF_p}^{\oplus r'_j} \oplus M\!\!N$ 
in the category $\NNum(k)_{\bbF_p}$. Moreover, $M\!\!N$ does not contains the noncommutative numerical motives $\{U(A_j)_{\bbF_p}\}_{1\leq j \leq m}$ and $\{U(A'_j)_{\bbF_p}\}_{1\leq j \leq m'}$ as direct summands. Consequently, \eqref{eq:isom-1} yields an isomorphism:
\begin{equation}\label{eq:iso-NNum}
\oplus_{j=1}^m U(A_j)_{\bbF_p}^{\oplus (r_j +1)} \oplus \oplus_{j=1}^{m'} U(A'_j)_{\bbF_p}^{\oplus r'_j} \oplus M\!\!N \simeq \oplus_{j=1}^m U(A_j)_{\bbF_p}^{\oplus r_j} \oplus \oplus_{j=1}^{m'} U(A'_j)_{\bbF_p}^{\oplus (r'_j+1)} \oplus M\!\!N\,.
\end{equation}

We claim that the following composition maps (with $1\leq i \leq m$ and $1\leq j \leq m'$)
\begin{equation}\label{eq:pairings}
\Hom_{\NNum(k)_{\bbF_p}}(U(A_i)_{\bbF_p}, M\!\!N)\times \Hom_{\NNum(k)_{\bbF_p}}(M\!\!N, U(A'_j)_{\bbF_p})\to \Hom_{\NNum(k)_{\bbF_p}}(U(A_i)_{\bbF_p}, U(A'_j)_{\bbF_p})
\end{equation}
are equal to zero; similarly with $A_i$ and $A'_j$ replaced by $A'_j$ and $A_i$, respectively. On the one hand, if $[A_i]^p\neq [A'_j]^p$, it follows from the above Proposition \ref{prop:p-Brauer}(ii) that the right-hand side of \eqref{eq:pairings} is equal to zero. On the other hand, if $[A_i]^p=[A'_j]^p$, it follows Proposition \ref{prop:p-Brauer} that $U(A_i)_{\bbF_p}\simeq U(A'_j)_{\bbF_p}$ in the category $\NNum(k)_{\bbF_p}$ and that the right-hand side of \eqref{eq:pairings} is isomorphic to $\bbF_p$. Since the category $\NNum(k)_{\bbF_p}$ is $\bbF_p$-linear and $M\!\!N$ does not contains the noncommutative numerical motives $\{U(A_j)_{\bbF_p}\}_{1\leq j \leq m}$ and $\{U(A'_j)_{\bbF_p}\}_{1\leq j \leq m'}$ as direct summands, we then conclude that the composition map \eqref{eq:pairings} is necessarily equal to zero; otherwise $M\!\!N$ would contain $U(A_i)_{\bbF_p}$, or equivalently $U(A'_j)_{\bbF_p}$, as a direct summand. 

Note that the triviality of the composition maps \eqref{eq:pairings} implies that the above isomorphism \eqref{eq:iso-NNum} in the category $\NNum(k)_{\bbF_p}$ restricts to an isomorphism 
\begin{equation}\label{eq:isom-3}
\oplus^m_{j=1}U(A_j)_{\bbF_p}^{\oplus (r_j +1)} \oplus \oplus^{m'}_{j=1} U(A'_j)_{\bbF_p}^{\oplus r'_j} \simeq \oplus_{j=1}^m U(A_j)_{\bbF_p}^{\oplus r_j} \oplus \oplus_{j=1}^{m'} U(A'_j)_{\bbF_p}^{\oplus (r'_j +1)} 
\end{equation}
in the full subcategory $\underline{\mathrm{CSA}}(k)_{\bbF_p}^\oplus$. Recall from Corollary \ref{cor:graded} that the category $\underline{\mathrm{CSA}}(k)_{\bbF_p}^\oplus$ is equivalent to the category $\mathrm{Vect}_{\mathrm{Br}(k)\{p\}}(k)$ of finite-dimensional $\mathrm{Br}(k)\{p\}$-graded $\bbF_p$-vector spaces. Under the equivalence \eqref{eq:equivalence-induced}, the isomorphism \eqref{eq:isom-3} corresponds to the following isomorphism:
\begin{equation}\label{eq:isom-33}
\oplus^m_{j=1} (\bbF_p^{\oplus (r_j +1)} [A_j]^p) \oplus \oplus^{m'}_{j=1} (\bbF_p^{\oplus r'_j}[A'_j]^p) \simeq \oplus_{j=1}^m (\bbF_p^{\oplus r_j} [A_j]^p) \oplus \oplus_{j=1}^{m'} (\bbF_p^{\oplus (r'_j +1)}[A'_j]^p) \,.
\end{equation}
Therefore, since the category $\mathrm{Vect}_{\mathrm{Br}(k)\{p\}}(k)$ has the Krull-Schmidt property, it follows from \eqref{eq:isom-33} that the finite-dimensional $\mathrm{Br}(k)\{p\}$-graded $\bbF_p$-vector spaces $\oplus_{j=1}^m (\bbF_p [A_j]^p)$ and $\oplus_{j=1}^{m'} (\bbF_p [A'_j]^p)$ are isomorphic, that $m=m'$, and that there exists a permutation $\sigma_p$ (which depends on $p$) such that $[A'_j]^p=[A_{\sigma_p(j)}]^p$ in $\mathrm{Br}(k)\{p\}$ for every $1 \leq j \leq m$. Consequently, the proof follows now from Proposition \ref{prop:key}.
\end{proof}
\begin{lemma}\label{lem:splitting}
We have $N\!\!M_{\bbF_p}\simeq \oplus^m_{j=1} U(A_j)_{\bbF_p}^{\oplus r_j} \oplus \oplus^{m'}_{j=1} U(A'_j)_{\bbF_p}^{\oplus r'_j} \oplus M\!\!N$ in the category $\NNum(k)_{\bbF_p}$ for some noncommutative numerical motive $M\!\!N$ and integers $r_j, r'_j\geq 0$. Moreover, $M\!\!N$ does not contains the noncommutative numerical motives $\{U(A_j)_{\bbF_p}\}_{1\leq j \leq m}$ and $\{U(A'_j)_{\bbF_p}\}_{1\leq j \leq m'}$ as direct summands.
\end{lemma}
\begin{proof}
By definition, the category $\NNum(k)_{\bbF_p}$ is idempotent complete. Therefore, by inductively splitting all the (possible) direct summands $\{U(A_j)_{\bbF_p}\}_{1\leq j \leq m}$ and $\{U(A'_j)_{\bbF_p}\}_{1\leq j \leq m'}$ of $N\!\!M_{\bbF_p}$, we obtain an isomorphism 
\begin{equation}\label{eq:split}
N\!\!M_{\bbF_p}\simeq \oplus^m_{j=1} U(A_j)_{\bbF_p}^{\oplus r_j} \oplus \oplus^{m'}_{j=1} U(A'_j)_{\bbF_p}^{\oplus r'_j} \oplus M\!\!N
\end{equation}
in the category $\NNum(k)_{\bbF_p}$ for some noncommutative numerical motive $M\!\!N$ and integers $r_j, r'_j\geq 0$. Note that the inductive splitting procedure stops at a finite stage. Otherwise, the following $\bbF_p$-vector spaces 
\begin{eqnarray}\label{eq:Homs}
\Hom_{\NNum(k)_{\bbF_p}}(U(A_i)_{\bbF_p},N\!\!M_{\bbF_p}) && \Hom_{\NNum(k)_{\bbF_p}}(U(A'_j)_{\bbF_p},N\!\!M_{\bbF_p})
\end{eqnarray}
would be infinite-dimensional, which is impossible because the $\bbF_p$-vector spaces $\Hom_{\NNum(k)}(U(A_i),N\!\!M)\otimes \bbF_p$ and $\Hom_{\NNum(k)}(U(A'_j), N\!\!M)\otimes \bbF_p$ are finite-dimensional (consult \cite[Thm.~1.2]{Separable}) and surject onto \eqref{eq:Homs}. 
 \end{proof}
\begin{notation}
Let us write $K_0(\mathrm{CSA}(k)^\oplus)$ for the Grothendieck ring of the additive symmetric monoidal category $\mathrm{CSA}(k)^\oplus$; consult Notation \ref{not:CSA}.
\end{notation}
\begin{proposition}\label{prop:cancellation2}
The inclusion of categories $\mathrm{CSA}(k)^\oplus \subset \NChow(k)$ gives rise to an injective ring homomorphism $K_0(\mathrm{CSA}(k)^\oplus) \to K_0(\NChow(k))$.
\end{proposition}
\begin{proof}
Recall first that the group completion of an arbitrary monoid $(M,+)$ is defined as the quotient of the product $M\times M$ by the following equivalence relation:
\begin{equation}\label{eq:relation}
(m,n) \sim (m',n'):= \exists \,r \in M\,\,\mathrm{such}\,\,\mathrm{that}\,\,m+n'+r=n+m'+r\,.
\end{equation}
Let us write $K_0(\NChow(k))^+$ for the semi-ring of the additive symmetric monoidal category $\NChow(k)$. Concretely, $K_0(\NChow(k))^+$ is the set of isomorphism classes of noncommutative Chow motives equipped with the addition, resp. multiplication, law induced by $\oplus$, resp. $\otimes$. In the same vein, let us write $K_0(\mathrm{CSA}(k)^\oplus)^+$ for the semi-ring of the additive symmetric monoidal category $\mathrm{CSA}(k)^\oplus$. Clearly, the inclusion of categories $\mathrm{CSA}(k)^\oplus \subset \NChow(k)$ gives rise to an injective homomorphism $K_0(\mathrm{CSA}(k)^\oplus)^+ \to K_0(\NChow(k))^+$. Therefore, by combining the above definition of group completion \eqref{eq:relation} with the cancellation Proposition \ref{prop:cancellation2}, we conclude that the induced (ring) homomorphism $K_0(\mathrm{CSA}(k)^\oplus) \to K_0(\NChow(k))$ is also injective.
\end{proof}

Consider the (composed) ring homomorphism
\begin{equation}\label{eq:composition-1}
K_0\mathrm{Var}(k)^{\mathrm{tw}} \subset K_0\mathrm{Var}(k) \stackrel{\mu_{\mathrm{nc}}}{\too} K_0(\NChow(k))\,.
\end{equation}
Let ${}_\gamma \cF$ be a twisted projective homogeneous variety with Jacques Tits central simple $k$-algebras $\{A_{\rho_i}\}_{1\leq j \leq n(\cF)}$. As proved in \cite[Thm.~2.1]{Homogeneous}, we have an isomorphism $U(\perf_\dg({}_\gamma \cF))\simeq\oplus_{i=1}^{n(\cF)}U(A_{\rho_i})$ in $\NChow(k)$. Therefore, making use of Proposition \ref{prop:cancellation2}, we conclude that the assignment $[{}_\gamma \cF]\mapsto [\oplus_{i=1}^{n(\cF)}U(A_{\rho_i})]= \Sigma_{i=1}^{n(\cF)}[U(A_{\rho_i})]$ gives rise to a (well-defined) motivic measure $K_0\mathrm{Var}(k)^{\mathrm{tw}} \to K_0(\mathrm{CSA}(k)^\oplus)$. Consequently, the proof of Theorem \ref{thm:main} follows now from the following result:
\begin{proposition}
The assignment $[A] \mapsto [U(A)]$ gives rise to a ring isomorphism $R_{\mathrm{B}}(k) \stackrel{\simeq}{\to} K_0(\mathrm{CSA}(k)^\oplus)$.
\end{proposition}
\begin{proof}
Let us write $K_0(\mathrm{CSA}(k)^{\oplus})^+$ for the semi-ring of the additive symmetric monoidal category $\mathrm{CSA}(k)^{\oplus}$. Concretely, $K_0(\mathrm{CSA}(k)^{\oplus})^+$ is the set of isomorphism classes of the category $\mathrm{CSA}(k)^{\oplus}$ equipped with the addition, resp. multiplication, law induced by $\oplus$, resp. $\otimes$. Consider also the semi-ring $\bbN[\mathrm{Br}(k)]$ and the following (semi-ring) homomorphism: 
\begin{eqnarray}\label{eq:homo-induced}
\bbN[\mathrm{Br}(k)] \too K_0(\mathrm{CSA}(k)^{\oplus})^+ && \Sigma_{j=1}^m [A_j] \mapsto \Sigma_{j=1}^m [U(A_j)]\,.
\end{eqnarray}
The homomorphism \eqref{eq:homo-induced} is surjective. 
Moreover, thanks to Proposition \ref{prop:key}, it yields an isomorphism
\begin{equation}\label{eq:induced}
\bbN[\mathrm{Br}(k)]/ \{\Sigma^m_{j=1}[A_j] = \Sigma^m_{j=1} [A'_j]\,|\,\, \forall\,p \,\,\, \exists\,\sigma_p\,\,\,[A'_j]^p=[A_{\sigma_p(j)}]^p\,\,\,\forall\,1\leq j \leq m  \} \stackrel{\simeq}{\too} K_0(\mathrm{CSA}(k)^{\oplus})^+\,,
\end{equation} 
where $p$ is a prime number and $\sigma_p$ a permutation of the set $\{1, \ldots, m\}$. Thanks to Lemma \ref{lem:relations} below, the left-hand side of \eqref{eq:induced} may be replaced by the following semi-ring 
$$\bbN[\mathrm{Br}(k)]/\{ [A''] + [A\otimes A' \otimes A'']= [A\otimes A''] + [A\otimes A'']\,|\,(\mathrm{ind}(A), \mathrm{ind}(A'))=1\}\,,$$ 
where $A$ and $A'$ are central simple $k$-algebras with coprime indexes. Therefore, by passing to group-completion, we obtain an induced (ring) isomorphism:
\begin{equation}\label{eq:induced1}
\bbZ[\mathrm{Br}(k)]/ \langle[A''] + [A\otimes A' \otimes A''] -[A\otimes A'] - [A\otimes A'']\,|\,(\mathrm{ind}(A), \mathrm{ind}(A'))=1 \rangle \stackrel{\simeq}{\too} K_0(\mathrm{CSA}(k)^{\oplus})\,.
\end{equation}
Finally, the proof follows now from the fact that the left-hand side of \eqref{eq:induced1} agrees with the ring $R_{\mathrm{B}}(k)$. 
\end{proof}
\begin{lemma}\label{lem:relations}
The following two sets of relations on the semi-ring $\bbN[\mathrm{Br}(k)]$ are equivalent:
\begin{equation}\label{eq:rel111}
\{\Sigma^m_{j=1}[A_j] = \Sigma^m_{j=1} [A'_j]\,|\,\, \forall\,p \,\,\, \exists\,\sigma_p\,\,\,[A'_j]^p=[A_{\sigma_p(j)}]^p\,\,\,\forall\,1\leq j \leq m \}
\end{equation}
\begin{equation}\label{eq:rel22}
\{ [A''] + [A\otimes A' \otimes A'']= [A\otimes A''] + [A\otimes A'']\,|\,(\mathrm{ind}(A), \mathrm{ind}(A'))=1\}\,.
\end{equation}
\end{lemma}
\begin{proof}
Note first that the set \eqref{eq:rel22} is contained in the set \eqref{eq:rel111}. Note also that since every permutation $\sigma_p$ can be written as a composition of transpositions, \eqref{eq:rel111} is equivalent to the following set of relations
 \begin{equation}\label{eq:rel3}
\{[A_1]+[A_2] = [A'_1]+[A'_2]\,\,|\,\, [A'_1]^q=[A_2]^q, [A'_2]^q=[A_1]^q\,\,\,\mathrm{and}\,\,\, [A'_1]^p=[A_1]^p, [A'_2]^p=[A_2]^p\,\,\forall\, p\neq q\}\,,
\end{equation}
where $q\neq p$ is a(ny) prime number. Therefore, it suffices to show that every relation in \eqref{eq:rel3} is a particular case of a relation in \eqref{eq:rel22}. Recall from the Artin-Wedderburn theorem that $A_1$, resp. $A_2$, may be written as the matrix algebra of a {\em unique} central division $k$-algebra $D_1$, resp. $D_2$. Let us write $D_1=\otimes_p D_1^p$, resp. $D_2=\otimes_p D_2^p$, for the associated $p$-primary decomposition; consult \cite[Prop.~4.5.16]{Gille}. By construction, we have $[D_1^p]=[D_1]^p=[A_1]^p$, resp. $[D_2^p]=[D_2]^p=[A_2]^p$. Now, consider the following central simple $k$-algebras:
\begin{eqnarray*}
A:=(D_1^q)^\op \otimes D_2^q & A':=(\otimes_{p\neq q} D_1^p)^\op \otimes(\otimes_{p\neq q} D_2^p) & A'':=A_1\,.
\end{eqnarray*}
Note that $(\mathrm{ind}(A),\mathrm{ind}(A'))=1$ and that for these choices the relation in \eqref{eq:rel22} reduces to the relation:
\begin{equation}\label{eq:relation-quasi}
[A_1] + [D_2] = [(\otimes_{p\neq q} D_1^p)\otimes D_2^q] + [D_1^q\otimes(\otimes_{p\neq q} D_2^p)]\,.
\end{equation}
Making use of the equalities $[D_2]=[A_2]$, $[(\otimes_{p\neq q} D_1^p)\otimes D_2^q]=[A'_1]$ and $[D_1^q\otimes(\otimes_{p\neq q} D_2^p)]=[A'_2]$, we hence conclude that \eqref{eq:relation-quasi} agrees with the relation in \eqref{eq:rel3}. This finishes the proof.
\end{proof}
\section{Properties of the Jacques Tits motivic measure}
\begin{proposition}\label{prop:auxiliar1}
Let $\{{}_\gamma \cF_j\}_{1\leq j \leq m}$ and $\{{}_{\gamma'} \cF'_j\}_{1\leq j \leq m'}$ be two families of twisted projective homogeneous varieties with Jacques Tits central simple $k$-algebras $\cup_{j=1}^m \{A_{\rho_{i_j}}\}_{1\leq i\leq n(\cF_j)}$ and $\cup_{j=1}^{m'}\{A'_{\rho'_{i_j}}\}_{1\leq i\leq n(\cF'_j)}$.
\begin{itemize}
\item[(i)] We have an isomorphism $U(\perf_\dg(\Pi^m_{j=1}{}_\gamma \cF_j))\simeq U(\perf_\dg(\Pi^{m'}_{j=1}{}_{\gamma'} \cF'_j))$ in the category $\NChow(k)$ if and only if $\mu_{\mathrm{JT}}([\Pi^m_{j=1}{}_\gamma \cF_j])=\mu_{\mathrm{JT}}([\Pi^{m'}_{j=1}{}_{\gamma'} \cF'_j])$ in $R_{\mathrm{B}}(k)$. 
\item[(ii)] If $\mu_{\mathrm{JT}}([\Pi^m_{j=1}{}_\gamma \cF_j])=\mu_{\mathrm{JT}}([\Pi^{m'}_{j=1}{}_{\gamma'} \cF'_j])$ in $R_{\mathrm{B}}(k)$, then $\langle \cup_{j=1}^m \{[A_{\rho_{i_j}}]\}_{1\leq i\leq n(\cF_j)} \rangle= \langle \cup_{j=1}^{m'}\{[A'_{\rho'_{i_j}}]\}_{1\leq i \leq n(\cF_j')} \rangle$.
\end{itemize}
\end{proposition}
\begin{proof}
Note first that an iterated application of Remark \ref{rk:generalization} leads to the following isomorphisms:
\begin{equation}\label{eq:ncmotives1}
U(\perf_\dg(\Pi^m_{j=1}{}_\gamma \cF_j))\simeq \otimes^m_{j=1} U(\perf_\dg({}_\gamma \cF_j))\simeq \otimes^m_{j=1} \oplus_{i=1}^{n(\cF_j)}U(A_{\rho_{i_j}}) 
\end{equation}
\begin{equation}\label{eq:ncmotives2}
U(\perf_\dg(\Pi^{m'}_{j=1}{}_{\gamma'} \cF'_j))\simeq \otimes^{m'}_{j=1} U(\perf_\dg({}_{\gamma'} \cF'_j))\simeq \otimes^{m'}_{j=1} \oplus_{i=1}^{n(\cF'_j)}U(A'_{\rho'_{i_j}}) \,.
\end{equation}
Suppose that we have an isomorphism $U(\perf_\dg(\Pi^m_{j=1}{}_\gamma \cF_j))\simeq U(\perf_\dg(\Pi^{m'}_{j=1}{}_{\gamma'} \cF'_j))$ in the category $\NChow(k)$. Since the noncommutative Chow motives \eqref{eq:ncmotives1}-\eqref{eq:ncmotives2} belong to the subcategory $\mathrm{CSA}(k)^{\oplus}$, it follows from the construction of the Jacques Tits motivic measure that $\mu_{\mathrm{JT}}([\Pi^m_{j=1}{}_\gamma \cF_j])=\mu_{\mathrm{JT}}([\Pi^{m'}_{j=1}{}_{\gamma'} \cF'_j])$ in $R_{\mathrm{B}}(k)$. Conversely, if $\mu_{\mathrm{JT}}([\Pi^m_{j=1}{}_\gamma \cF_j])=\mu_{\mathrm{JT}}([\Pi^{m'}_{j=1}{}_{\gamma'} \cF'_j])$ in $R_{\mathrm{B}}(k)$, then it follows from the construction of the Jacques Tits motivic measure that $[U(\perf_\dg(\Pi^m_{j=1}{}_\gamma \cF_j))]=[U(\perf_\dg(\Pi^{m'}_{j=1}{}_{\gamma'} \cF'_j))]$ in $K_0(\NChow(k))$. By definition of $K_0(\NChow(k))$, this implies that there exists a noncommutative Chow motive $N\!\!M$ such that 
\begin{equation}\label{eq:iso-main111}
U(\perf_\dg(\Pi^m_{j=1}{}_\gamma \cF_j))\oplus N\!\!M\simeq U(\perf_\dg(\Pi^{m'}_{j=1}{}_{\gamma'} \cF'_j))\oplus N\!\!M\,.
\end{equation} 
Since the noncommutative Chow motives \eqref{eq:ncmotives1}-\eqref{eq:ncmotives2} belong to the subcategory $\mathrm{CSA}(k)^{\oplus}$, we hence conclude from the cancellation Proposition \ref{prop:cancellation} (applied to \eqref{eq:iso-main111}) that $U(\perf_\dg(\Pi^m_{j=1}{}_\gamma \cF_j))\simeq U(\perf_\dg(\Pi^{m'}_{j=1}{}_{\gamma'} \cF'_j))$ in the category $\NChow(k)$. This proves item (i). Item (ii) follows now from item (i), from the fact that the noncommutative Chow motives \eqref{eq:ncmotives1}-\eqref{eq:ncmotives2} belong to the subcategory $\mathrm{CSA}(k)^{\oplus}$, and from Corollary \ref{cor:key}.
\end{proof}
\begin{lemma}
The assignment $[{}_\gamma \cF] \mapsto n(\cF)$ gives rise to a motivic measure $\mu_\rho\colon K_0\mathrm{Var}(k)^{\mathrm{tw}} \to \bbZ$.
\end{lemma}
\begin{proof}
The ring $R_{\mathrm{B}}(k)$ comes equipped with the augmentation $\Sigma^m_{j=1} n_j [A_j] \mapsto \Sigma_{j=1}^m n_j$. By pre-composing this augmentation with the Jacques Tits motivic measure $\mu_{\mathrm{JT}}$, we hence obtain the motivic measure $\mu_\rho$.
\end{proof}
\section{Proof of Theorem \ref{thm:app1}}
{\bf Item (i).} Recall first that $\mathrm{dim}(\mathrm{SB}(A))=\mathrm{deg}(A)-1$. Following \S\ref{sub:Severi}, note that $\mu_\rho([\mathrm{SB}(A)])=\mathrm{deg}(A)$. Therefore, if $[\mathrm{SB}(A)]=[\mathrm{SB}(A')]$ in the Grothendieck ring of varieties $K_0\mathrm{Var}(k)$, we conclude that $\mathrm{deg}(A)=\mathrm{deg}(A')$. This is equivalent to the equality $\mathrm{dim}(\mathrm{SB}(A))=\mathrm{dim}(\mathrm{SB}(A'))$.

\smallskip

{\bf Item (ii).} If $[\mathrm{SB}(A)]=[\mathrm{SB}(A')]$ in the Grothendieck ring of varieties $K_0\mathrm{Var}(k)$, then it follows from Proposition \ref{prop:auxiliar1}(ii) that $\langle [A] \rangle=\langle [A']\rangle$. In particular, we have $\mathrm{per}(A)=\mathrm{per}(A')$.

\smallskip

{\bf Item (iii).} When $[A]\in {}_2\mathrm{Br}(k)$, it follows from the equality $\langle [A]\rangle = \langle [A']\rangle$ (proved in item (ii)) that $[A]=[A']$ in the Brauer group $\mathrm{Br}(k)$. Using the fact that $\mathrm{deg}(A)=\mathrm{deg}(A')$ (proved in item (i)), we hence conclude that $A\simeq A'$. This implies that $\mathrm{SB}(A)\simeq \mathrm{SB}(A')$.

\smallskip

{\bf Item (iv).} When $\mathrm{per}(A)\in \{3, 4, 5, 6\}$, it follows from the equality $\langle [A]\rangle =\langle [A']\rangle$ (proved in item (ii)) that $[A]$ is equal to $[A']$, to $-[A']$, to $2[A']$, or to $-2[A']$. In the case where $[A]=[A']$, a proof similar to the one of item (iii) implies that $\mathrm{SB}(A)\simeq \mathrm{SB}(A')$. In the case where $[A]=-[A']$, the Amitsur conjecture holds thanks to the work of Roquette \cite{Roquette}. Hence, following item (iv') of Remark \ref{rk:Amitsur}, we conclude that $\mathrm{SB}(A)$ and $\mathrm{SB}(A')$ are birational to each other. In the case where $[A]=2[A']$ (or $[A]=-2[A']$), the Amitsur conjecture holds thanks to the work of Tregub \cite{Tregub}. Hence, we conclude similarly that $\mathrm{SB}(A)$ and $\mathrm{SB}(A')$ are birational~to~each~other.
\section{Proof of Proposition \ref{prop:app2}}
{\bf Item (i).} If $[C(a,b)\times C(a',b')]=[C(a'',b'')\times C(a''',b''')]$ in the Grothendieck ring of varieties $K_0\mathrm{Var}(k)$, then it follows from Proposition \ref{prop:auxiliar1}(ii) that $\langle [Q], [Q']\rangle =\langle [Q''],[Q''']\rangle$. Since $[Q], [Q'], [Q''], [Q'''] \in {}_2\mathrm{Br}(k)$, the latter equality implies that $[Q]$ (or $[Q']$) is equal to $[Q'']$ or to $[Q''']$. Using the fact that $\mathrm{deg}(Q)=\mathrm{deg}(Q')=\mathrm{deg}(Q'')=\mathrm{deg}(Q''')=2$, we hence conclude that $Q$ (or $Q'$) is isomorphic to $Q''$ or to $Q'''$. Equivalently, the conic $C(a,b)$ (or $C(a',b')$) is isomorphic to $C(a'',b'')$ or to $C(a''',b''')$. 

\smallskip

{\bf Item (ii).} When the quaternion $k$-algebras $Q$ and $Q'$ are unlinked, i.e., when $Q\otimes Q'$ is a division $k$-algebra, we have $\mathrm{ind}(Q\otimes Q')=4$. Consequently, using the fact that $\mathrm{ind}(Q)=\mathrm{ind}(Q')=\mathrm{ind}(Q'')=\mathrm{ind}(Q''')=2$, we conclude from the equality $\langle [Q], [Q']\rangle =\langle [Q''],[Q''']\rangle$ (proved in item (i)) that $[Q]=[Q'']$ and $[Q']=[Q''']$ or that $[Q]=[Q''']$ and $[Q']=[Q'']$. Since $\mathrm{deg}(Q)=\mathrm{deg}(Q')=\mathrm{deg}(Q'')=\mathrm{deg}(Q''')=2$, this implies that $Q\simeq Q''$ and $Q'\simeq Q'''$ or that $Q\simeq Q'''$ and $Q'\simeq Q''$. Equivalently, we have $C(a,b)\simeq C(a'',b'')$ and $C(a',b')\simeq C(a''',b''')$ or $C(a',b')\simeq C(a''',b''')$ and $C(a',b')\simeq C(a'',b'')$. In both cases we have an isomorphism $C(a,b)\times C(a',b')\simeq C(a'',b'') \times C(a''',b''')$.
\section{Proof of Theorem \ref{thm:app3}}
{\bf Item (i).} Recall first that $\mathrm{dim}(\mathrm{Gr}(d;A))=d\times(\mathrm{deg}(A)-d)$. Following \S\ref{sub:Grass}, note that $\mu_\rho([\mathrm{Gr}(d;A)])=\binom{\mathrm{deg}(A)}{d}$ because this is the number of Young diagrams inside the rectangle with $d$ lines and $\mathrm{deg}(A)-d$ columns. Hence, if $[\mathrm{Gr}(d;A)]=[\mathrm{Gr}(d';A')]$ in the Grothendieck ring of varieties $K_0\mathrm{Var}(k)$, we conclude that $\binom{\mathrm{deg}(A)}{d}=\binom{\mathrm{deg}(A')}{d'}$. By definition of the binomial coefficients, this implies that $\mathrm{deg}(A)=\mathrm{deg}(A')$ and that $d=d'$ or that $d'=\mathrm{deg}(A)-d$. In both cases, we have $\mathrm{dim}(\mathrm{Gr}(d;A))=\mathrm{dim}(\mathrm{Gr}(d';A'))$.

\smallskip

{\bf Item (ii).} If $[\mathrm{Gr}(d;A)]=[\mathrm{Gr}(d';A')]$ in the Grothendieck ring of varieties $K_0\mathrm{Var}(k)$, then it follows from Proposition \ref{prop:auxiliar1}(ii) that $\langle [A] \rangle=\langle [A']\rangle$. In particular, we have $\mathrm{per}(A)=\mathrm{per}(A')$.

\smallskip

{\bf Item (iii).} When $[A]\in {}_2\mathrm{Br}(k)$, it follows from the equality $\langle [A]\rangle = \langle [A']\rangle$ (proved in item (ii)) that $[A]=[A']$ in the Brauer group $\mathrm{Br}(k)$. Using the fact that $\mathrm{deg}(A)=\mathrm{deg}(A')$ (proved in item (i)), we hence conclude that $A\simeq A'$. Since $d=d'$ or $d=\mathrm{deg}(A)-d'$, this implies that $\mathrm{Gr}(d;A)\simeq \mathrm{Gr}(d';A')$.
\section{Proof of Theorem \ref{thm:app4}}
{\bf Item (i).} Recall first that $\mathrm{dim}(Q_q)=n-2$. Following \S\ref{sub:quadrics}, note that $\mu_\rho(Q_q)$ is equal to $n$ when $n$ is even or to $n-1$ when $n$ is odd. Hence, if $[Q_q]=[Q_{q'}]$ in the Grothendieck ring of varieties $K_0\mathrm{Var}(k)$, we conclude that $n=n'$ when $n$ is even or that $n-1=n'-1$ when $n$ is odd. In both cases, we have $\mathrm{dim}(Q_q)=\mathrm{dim}(Q_{q'})$. This is equivalent to the equality $n=n'$.

\smallskip

{\bf Item (ii).} If $[Q_q]=[Q_{q'}]$ in the Grothendieck ring of varieties $K_0\mathrm{Var}(k)$, then it follows from Proposition \ref{prop:auxiliar1}(ii) that $\langle [C^+_0(q)]\rangle = \langle [C^+_0(q')]\rangle$ when $n$ is even or $\langle [C_0(q)]\rangle = \langle [C_0(q')]\rangle$ when $n$ is odd. Since $[C^+_0(q)], [C_0(q)] \in {}_2\mathrm{Br}(k)$, this implies that $[C^+_0(q)]=[C^+_0(q')]$ when $n$ is even or $[C_0(q)]=[C_0(q')]$ when $n$ is odd. Using the fact that $\mathrm{deg}(C_0^+(q))=2^{\frac{n}{2}-1}$, $\mathrm{deg}(C_0(q))=2^{\lfloor\frac{n}{2}\rfloor}$, and $n=n'$ (proved in item (i)), we hence conclude that $C^+_0(q)\simeq C^+_0(q')$ when $n$ is even or $C_0(q)\simeq C_0(q')$ when $n$ is odd.

\smallskip

{\bf Item (iii).} When $n=6$, the assignment $q\mapsto C_0^+(q)$ gives rise to a one-to-one correspondence between similarity classes of non-degenerate quadratic forms with trivial discriminant of dimension $6$ and isomorphism classes of quaternion algebras; consult \cite[Cor.~15.33]{Book-inv}. Consequently, the proof follows from the combination of item (ii) with the general fact that two quadratic forms $q$ and $q'$ are similar if and only if the associated quadric hypersurfaces $Q_q$ and $Q_{q'}$ are isomorphic.

\smallskip

{\bf Item (iv).} When $I^3(k)=0$, we have the following classification result: if $n=n'$ and $C^+_0(q)\simeq C^+_0(q')$ when $n$ is even or $C_0(q)\simeq C_0(q')$ when $n$ is odd, then the quadratic forms $q$ and $q'$ are similar; consult \cite[Thm.~3']{ELam}. Consequently, the proof follows from the combination of items (i)-(ii) with the general fact that two quadratic forms $q$ and $q'$ are similar if and only if the quadric hypersurfaces $Q_q$ and $Q_{q'}$ are isomorphic.

\section{Proof of items $\mathrm{(i)}$-$\mathrm{(ii)}$ of Theorem \ref{thm:prod-quadrics}}
{\bf Item (i).} Recall first that $\mathrm{dim}(\Pi_{j=1}^m Q_{q_j})= \Sigma_{j=1}^m \mathrm{dim}(Q_{q_j})=m\times (n-2)$. Following \S\ref{sub:quadrics}, note that $\mu_\rho([\Pi_{j=1}^m Q_{q_j}])=\Pi_{j=1}^m \mu_\rho([Q_{q_j}])$ is equal to $n^m$ when $n$ is even or to $(n-1)^m$ when $n$ is odd. Hence, if $[\Pi_{j=1}^m Q_{q_j}]=[\Pi_{j=1}^{m'} Q_{q'_j}]$ in the Grothendieck ring of varieties $K_0\mathrm{Var}(k)$, we conclude that $n^m=n^{m'}$ when $n$ is even or that $(n-1)^m=(n-1)^{m'}$ when $n$ is odd. In both cases, we have $\mathrm{dim}(\Pi_{j=1}^m Q_{q_j})=\mathrm{dim}(\Pi_{j=1}^{m'} Q_{q'_j})$. This is equivalent to the equality $m=m'$.

\smallskip

{\bf Item (ii).} If $[\Pi_{j=1}^m Q_{q_j}]=[\Pi_{j=1}^{m'}Q_{q'_j}]$ in the Grothendieck ring of varieties $K_0\mathrm{Var}(k)$, then it follows from Proposition \ref{prop:auxiliar1}(ii) that $\langle \{[C^+_0(q_j)]\}_{1\leq j\leq m} \rangle=\langle \{[C^+_0(q'_j)]\}_{1\leq j\leq m'} \rangle$ when $n$ is even or $\langle \{[C_0(q_j)]\}_{1\leq j\leq m} \rangle=\langle \{[C_0(q'_j)]\}_{1\leq j\leq m'} \rangle$ when $n$ is odd.
\section{Proof of items $\mathrm{(iii)}$-$\mathrm{(iv)}$-$\mathrm{(iv')}$ of Theorem \ref{thm:prod-quadrics}}
We start with some auxiliar results of independent interest:
\begin{proposition}\label{prop:3-equivalences}
Let $q$ and $q'$ be two non-degenerate quadratic forms with trivial discriminant of dimension $n$. When $n=6$ or $I^3(k)=0$, the following conditions are equivalent:
\begin{itemize}
\item[(i)] We have $Q_q\simeq Q_{q'}$.
\item[(ii)] We have an isomorphism $U(\perf_\dg(Q_q))\simeq U(\perf_\dg(Q_{q'}))$ in the category $\NChow(k)$.
\item[(iii)] We have $[C_0^+(q)]=[C_0^+(q')]$ when $n$ is even or $[C_0(q)]=[C_0(q')]$ when $n$ is odd.
\end{itemize}
\end{proposition}
\begin{proof}
The implication (i) $\Rightarrow$ (ii) is clear. If we have an isomorphism $U(\perf_\dg(Q_q))\simeq U(\perf_\dg(Q_{q'})$ in the category $\NChow(k)$, then it follows from Proposition \ref{prop:auxiliar1} that $\langle [C_0^+(q)]\rangle = \langle [C_0^+(q')]\rangle$ when $n$ is even or $\langle [C_0(q)]\rangle = \langle [C_0(q')]\rangle$ when $n$ is odd. Using the fact that $[C_0^+(q)], [C_0(q)] \in {}_2\mathrm{Br}(k)$, we hence conclude that $[C^+_0(q)]=[C_0^+(q')]$ when $n$ is even or $[C_0(q)]=[C_0(q')]$ when $n$ is odd. This proves the implication (ii) $\Rightarrow$ (iii). In what concerns the implication (iii) $\Rightarrow$ (i), recall that $\mathrm{deg}(C_0^+(q))=\mathrm{deg}(C_0^+(q'))=2^{\frac{n}{2}-1}$ and that $\mathrm{deg}(C_0(q))=\mathrm{deg}(C_0(q'))=2^{\lfloor\frac{n}{2}\rfloor}$. Therefore, if $[C_0^+(q)]=[C_0^+(q')]$ when $n$ is even or $[C_0(q)]=[C_0(q')]$ when $n$ is odd, we have $C_0^+(q)\simeq C_0^+(q')$ when $n$ is even or $C_0(q)\simeq C_0(q')$ when $n$ is odd. When $n=6$, we hence conclude from the one-to-one correspondence $q\mapsto C_0^+(q)$ between similarity classes of non-degenerate quadratic forms with trivial discriminant of dimension $6$ and isomorphism classes of quaternion algebras (consult \cite[Cor.~15.33]{Book-inv}), that the quadratic forms $q$ and $q'$ are similar or, equivalently, that $Q_q\simeq Q_{q'}$. In the same vein, when $I^3(k)=0$, we conclude from the classification result \cite[Thm.~3']{ELam} (which asserts that if $C_0^+(q)\simeq C_0^+(q')$ when $n$ is even or $C_0(q)\simeq C_0(q')$ when $n$ is odd, then the quadratic forms $q$ and $q'$ are similar) that the quadratic forms $q$ and $q'$ are similar or, equivalently, that $Q_q\simeq Q_{q'}$.
\end{proof}
The next result may be understood as the $\otimes$-analogue of Proposition \ref{prop:cancellation}.
\begin{proposition}[$\otimes$-cancellation]\label{prop:cancellation22}
Let $\{A_j\}_{1\leq j \leq m}$ and $\{A'_j\}_{1\leq j \leq m'}$ be two families of central simple $k$-algebras and $q$ a non-degenerate quadratic form with trivial discriminant of dimension $n\geq 5$. If we have an isomorphism $\oplus_{j=1}^m U(A_j)\otimes U(\perf_\dg(Q_q))\simeq \oplus_{j=1}^{m'} U(A'_j) \otimes U(\perf_\dg(Q_q))$ in the category $\NChow(k)$, then $m=m'$ and $\oplus_{j=1}^m U(A_j)\simeq \oplus_{j=1}^{m} U(A'_j)$.
\end{proposition}
\begin{proof}
We prove first the case where $n\geq 6$ is even. Recall from \cite[Example~3.8]{Homogeneous} that, since the central simple $k$-algebras $C_0^+(q)$ and $C_0^-(q)$ are isomorphic, we have the following computation:
$$U(\perf_\dg(Q_q))\simeq U(k)^{\oplus (n-2)} \oplus U(C_0^+(q)) \oplus U(C_0^-(q)) \simeq U(k)^{\oplus (n-2)} \oplus U(C_0^+(q))^{\oplus 2}\,.$$
Consequently, we obtain the following computation:
\begin{equation}\label{eq:computation111}
\oplus^m_{j=1} U(A_j) \otimes U(\perf_\dg(Q_q))\simeq \oplus^m_{j=1} U(A_j)^{\oplus (n-2)} \oplus \oplus^m_{j=1} U(A\otimes C_0^+(q))^{\oplus 2}\,.
\end{equation}
Making use of \eqref{eq:computation111}, the given isomorphism $ \oplus_{j=1}^{m} U(A_j) \otimes U(\perf_\dg(Q_q))\simeq \oplus_{j=1}^{m'} U(A'_j) \otimes U(\perf_\dg(Q_q))$ in the category $\NChow(k)$ may then be re-written as the following isomorphism
\begin{equation}\label{eq:iso-main}
\oplus^m_{j=1} U(A_j)^{\oplus (n-2)} \oplus \oplus^m_{j=1} U(A_j \otimes C_0^+(q))^{\oplus 2} \simeq \oplus^{m'}_{j=1} U(A'_j)^{\oplus (n-2)} \oplus \oplus^{m'}_{j=1} U(A'_j \otimes C_0^+(q))^{\oplus 2}
\end{equation}
in the category $\mathrm{CSA}(k)^{\oplus}$. Therefore, by applying Proposition \ref{prop:key} to the isomorphism \eqref{eq:iso-main}, we conclude that $(n-2)m+2m=(n-2)m'+2m'$, which implies that $m=m'$. In order to prove that $\oplus_{j=1}^m U(A_j)\simeq \oplus_{j=1}^m U(A'_j)$, we will also make use of Proposition \ref{prop:key}. Concretely, we need to show that for every prime number $p$ the following two sets of Brauer classes
\begin{eqnarray}\label{eq:sets0}
\{[A_1]^p, \ldots, [A_m]^p\} && \{[A'_1]^p, \ldots, [A'_m]^p\}
\end{eqnarray}
are the same up to permutation. Recall that $[C_0^+(q)] \in {}_2\mathrm{Br}(k)$. Therefore, when $p\neq 2$, the isomorphism \eqref{eq:iso-main} combined with Proposition \ref{prop:key} implies that the following two sets are the same up to permutation
\begin{eqnarray*}
\{\underbrace{[A_1]^p, \ldots, [A_m]^p}_{(n-2)}, \underbrace{[A_1]^p, \ldots, [A_m]^p}_{2}\} && \{\underbrace{[A'_1]^p, \ldots, [A'_m]^p}_{(n-2)}, \underbrace{[A'_1]^p, \ldots, [A'_m]^p}_{2}\} \,,
\end{eqnarray*}
where the numbers below the parenthesis denote the number of copies. Clearly, this implies that the above sets \eqref{eq:sets0} are also the same up to permutation. When $p=2$, the isomorphism \eqref{eq:iso-main} combined with Proposition \ref{prop:key} implies that the following two sets are the same up to permutation:
\begin{eqnarray}
\{\underbrace{[A_1]^2, \ldots, [A_m]^2}_{(n-2)}, \underbrace{[A_1\otimes C_0^+(q)]^2, \ldots, [A_m \otimes C_0^+(q)]^2}_{2}\} \label{eq:sets1}  \\
\{\underbrace{[A'_1]^2, \ldots, [A'_m]^2}_{(n-2)}, \underbrace{[A'_1\otimes C_0^+(q)]^2, \ldots, [A'_m\otimes C_0^+(q)]^2}_{2}\}\,.  \label{eq:sets11}
\end{eqnarray}
Note that in the case where $[C_0^+(q)]=[k]$, this also implies that the sets \eqref{eq:sets0} are the same up to permutation. Let us then assume that $[C_0^+(q)]\neq [k]$. In this case, each one of the sets \eqref{eq:sets1}-\eqref{eq:sets11} is equipped with a non-trivial involution induced by tensoring with $C_0^+(q)$ (we are implicitly ignoring the number of copies of each Brauer class). In particular, we have $[A_j \otimes C_0^+(q)]^2\neq [A_j]^2$ for every $1\leq j \leq m$ and if there exist integers $r$ and $s$ such that $[A_r \otimes C_0^+(q)]^2=[A_s]^2$, then $[A_r]^2=[A_s\otimes C_0^+(q)]^2$. Consequently, there exist disjoint subsets $\{j_1, \underline{j}_1, \ldots, j_r, \underline{j}_r\}$ and $\{i_1, \ldots, i_s\}$ of the set $\{1, \ldots, m\}$ and integers $n_1, \underline{n}_1, \ldots, n_r, \underline{n}_r\geq 1$ and $l_1, \ldots, l_s\geq 1$ such that \eqref{eq:sets1} agrees with the following set of {\em distinct} Brauer classes:
\begin{equation}\label{eq:sets3}
\cup_{t=1}^r \{\!\!\!\!\!\!\underbrace{[A_{j_t}]^2}_{(n-2)n_t+2\underline{n}_t}\!\!\!\!\!\!,\!\!\!\!\!\!\overbrace{[A_{\underline{j}_t}]^2}^{(n-2)\underline{n}_t+2n_t}\!\!\!\!\!\!\} \cup \cup_{t=1}^s \{\underbrace{[A_{i_t}]^2}_{(n-2)l_t}, \underbrace{[A_{i_t}\otimes C_0^+(q)]^2}_{2l_t}\}\,.
\end{equation}
Similarly, there exist disjoint subsets $\{j'_1, \underline{j}'_1, \ldots, j'_{r'}, \underline{j}'_{r'}\}$ and $\{i'_1, \ldots, i'_{s'}\}$ of the set $\{1, \ldots, m\}$ and integers $n'_1, \underline{n}'_1, \ldots, n'_{r'}, \underline{n}'_{r'}\geq 1$ and $l'_1, \ldots, l'_{s'}\geq 1$ such that \eqref{eq:sets11} agrees with the set of {\em distinct} Brauer classes:
\begin{equation}\label{eq:sets4}
\cup_{t=1}^{r'} \{\!\!\!\!\!\!\underbrace{[A'_{j'_t}]^2}_{(n-2)n'_t+2\underline{n}'_t}\!\!\!\!\!\!,\!\!\!\!\!\!\overbrace{[A'_{\underline{j}'_t}]^2}^{(n-2)\underline{n}'_t+2n'_t}\!\!\!\!\!\!\} \cup \cup_{t=1}^{s'} \{\underbrace{[A'_{i'_t}]^2}_{(n-2)l'_t}, \underbrace{[A'_{i'_t}\otimes C_0^+(q)]^2}_{2l'_t}\}\,.
\end{equation}
Note that, under these notations, the non-trivial involution on \eqref{eq:sets3} (we are implicitly ignoring the number of copies of each Brauer class) interchanges $[A_{j_t}]^2$ with $[A_{\underline{j}_t}]^2$ and $[A_{i_t}]^2$ with $[A_{i_t}\otimes C_0^+(q)]^2$; similarly for the non-trivial involution on \eqref{eq:sets4}. Note also that, under these notations, the above sets \eqref{eq:sets0} correspond to the following sets of Brauer classes (the Brauer classes in each one of the two sets are distinct):
\begin{eqnarray}\label{eq:distinct}
\cup_{t=1}^{r} \{\underbrace{[A_{j_t}]^2}_{n_t}, \underbrace{[A_{\underline{j}_t}]^2}_{\underline{n}_t}\} \cup \cup_{t=1}^{s} \{\underbrace{[A_{i_t}]^2}_{l_t}\} && \cup_{t=1}^{r'} \{\underbrace{[A'_{j'_t}]^2}_{n'_t}, \underbrace{[A'_{\underline{j}'_t}]^2}_{\underline{n}'_t}\} \cup \cup_{t=1}^{s'} \{\underbrace{[A'_{i'_t}]^2}_{l'_t}\}\,.  
\end{eqnarray}
Now, recall that the above sets \eqref{eq:sets3}-\eqref{eq:sets4} are the same up to permutation. In other words, there exists a permutation with identifies the distinct Brauer classes of the set \eqref{eq:sets3} with the distinct Brauer classes of the set \eqref{eq:sets4}. Therefore, making use of the non-trivial involutions on \eqref{eq:sets3}-\eqref{eq:sets4}, of the precise number of copies of each Brauer class, and of the assumption that $n\geq 6$ is even, we hence conclude that the following sets of Brauer classes are the same up to permutation:
\begin{equation}\label{eq:sets6}
\cup_{t=1}^r \{\!\!\!\!\!\!\underbrace{[A_{j_t}]^2}_{(n-2)n_t+2\underline{n}_t}\!\!\!\!\!\!,\!\!\!\!\!\!\overbrace{[A_{\underline{j}_t}]^2}^{(n-2)\underline{n}_t+2n_t}\!\!\!\!\!\!\}  \,\,\,\,  \cup_{t=1}^{r'} \{\!\!\!\!\!\!\underbrace{[A'_{j'_t}]^2}_{(n-2)n'_t+2\underline{n}'_t}\!\!\!\!\!\!,\!\!\!\!\!\! \overbrace{[A'_{\underline{j}'_t}]^2}^{(n-2)\underline{n}'_t+2n'_t}\!\!\!\!\!\!\}\quad \mathrm{resp.} \quad \cup_{t=1}^s \{\underbrace{[A_{i_t}]^2}_{(n-2)l_t}\} \,\,\,  \cup_{t=1}^{s'} \{\underbrace{[A'_{i'_t}]^2}_{(n-2)l'_t}\}\,.
\end{equation}
Making use once again of the non-trivial involution on the left-hand side of \eqref{eq:sets6} and of the precise number of copies of each Brauer class, we observe that $r=r'$ and that the sets $\{n_1, \underline{n}_1, \ldots, n_r, \underline{n}_r\}$ and $\{n'_1, \underline{n}'_1, \ldots, n'_{r'}, \underline{n}'_{r'}\}$ are the same up to permutation. In the same vein, we conclude from the right-hand side of \eqref{eq:sets6} that $s=s'$ and that the sets $\{l_1, \ldots, l_s\}$ and $\{l'_1, \ldots, l'_{s'}\}$ are the same up to permutation. This implies that the following sets of Brauer classes are the same up to permutation:
\begin{equation}\label{eq:sets8}
\cup_{t=1}^r \{\underbrace{[A_{j_t}]^2}_{n_t}, \underbrace{[A_{\underline{j}_t}]^2}_{\underline{n}_t}\}  \,\,\,\, \cup_{t=1}^{r'} \{\underbrace{[A'_{j'_t}]^2}_{n'_t}, \underbrace{[A'_{\underline{j}'_t}]^2}_{\underline{n}'_t}\} \quad \mathrm{resp.} \quad \cup_{t=1}^s \{\underbrace{[A_{i_t}]^2}_{l_t}\} \,\,\,  \cup_{t=1}^{s'} \{\underbrace{[A'_{i'_t}]^2}_{l'_t}\}\,.
\end{equation} 
Consequently, by concatenating the permutations provided by \eqref{eq:sets8}, we hence obtain a permutation which identifies the left-hand side of \eqref{eq:distinct} with the right-hand side of \eqref{eq:distinct}. In other words, the two sets in \eqref{eq:distinct} are the same up to permutation. This proves the case where $n\geq 6$ is even. The proof of the case where $n\geq 5$ is odd is similar: simply replace the above isomorphism $U(\perf_\dg(Q_q))\simeq U(k)^{\oplus (n-2)} \oplus U(C_0^+(q))^{\oplus 2}$ by the isomorphism $U(\perf_\dg(Q_q))\simeq U(k)^{\oplus (n-2)} \oplus U(C_0(q))$ and perform all the subsequent computations.
\end{proof}
\begin{lemma}\label{lem:implications}
Given integers $n \geq 5$ and $m,l\geq 0$, we have the implications (consult Notation \ref{not:sums}):
\begin{equation}\label{eq:implications}
\begin{cases}
(\Sigma^1_{\mathrm{even}}(m,n,l)> \Sigma^2_{\mathrm{even}}(m,n,l)) \Rightarrow (\Sigma^1_{\mathrm{even}}(m-1,n,l)> \Sigma^2_{\mathrm{even}}(m-1,n,l)) & \mathrm{n}\,\,\mathrm{even} \\
(\Sigma^1_{\mathrm{odd}}(m,n,l)> \Sigma^2_{\mathrm{odd}}(m,n,l)) \Rightarrow (\Sigma^1_{\mathrm{odd}}(m-1,n,l)> \Sigma^2_{\mathrm{odd}}(m-1,n,l)) & \mathrm{n}\,\,\mathrm{odd}\,.
\end{cases}
\end{equation}
\end{lemma}
\begin{proof}
Consider the following notations:
\begin{eqnarray*}
\Sigma^{1,1}_{\mathrm{even}}(m,n,l) & := &  \Sigma^{\lfloor l/2 \rfloor}_{r=0} \big(\tbinom{l}{2r} \times 2^{2r+1} \times (n-2)^{m-(2r+1)}\big) \\
\Sigma^{1,2}_{\mathrm{even}}(m,n,l) & := & \Sigma^{\lfloor l/2 \rfloor}_{r=0} \big(\tbinom{l}{2r+1} \times 2^{m-l+ (2r+1)} \times (n-2)^{l-(2r+1)} \big) \\
\Sigma^{1,1}_{\mathrm{odd}}(m,n,l) & := & \Sigma^{\lfloor l/2 \rfloor}_{r=0} \big(\tbinom{l}{2r} \times (n-2)^{m-(2r+1)}\big)  \\
\Sigma^{1,2}_{\mathrm{odd}}(m,n,l) & := & \Sigma^{\lfloor l/2 \rfloor}_{r=0} \big(\tbinom{l}{2r+1} \times (n-2)^{l-(2r+1)} \big) \,.
\end{eqnarray*}
Note that $\Sigma^1_{\mathrm{even}}(m,n,l)= \Sigma^{1,1}_{\mathrm{even}}(m,n,l) + \Sigma^{1,2}_{\mathrm{even}}(m,n,l)$ and $\Sigma^1_{\mathrm{odd}}(m,n,l)= \Sigma^{1,1}_{\mathrm{odd}}(m,n,l) + \Sigma^{1,2}_{\mathrm{odd}}(m,n,l)$. Note also that we have the following relations:
\begin{eqnarray*}
\Sigma^{1,1}_{\mathrm{even}}(m-1,n,l)= \frac{\Sigma^{1,1}_{\mathrm{even}}(m,n,l)}{(n-2)} & \Sigma^{1,1}_{\mathrm{odd}}(m-1,n,l)= \frac{\Sigma^{1,1}_{\mathrm{odd}}(m,n,l)}{(n-2)} &
\Sigma^{1,2}_{\mathrm{even}}(m-1,n,l)= \frac{\Sigma^{1,2}_{\mathrm{even}}(m,n,l)}{(n-2)} \\
 \Sigma^{1,2}_{\mathrm{odd}}(m-1,n,l)= \Sigma^{1,2}_{\mathrm{odd}}(m,n,l) &
 \Sigma^{2}_{\mathrm{even}}(m-1,n,l)= \frac{\Sigma^{2}_{\mathrm{even}}(m,n,l)}{(n-2)} 
&
\Sigma^{2}_{\mathrm{odd}}(m-1,n,l)= \frac{\Sigma^{2}_{\mathrm{odd}}(m,n,l)}{(n-2)} \,.
\end{eqnarray*}
By combining them, we obtain the above implications \eqref{eq:implications}.
\end{proof}
We now have the ingredients necessary to prove items (iii)-(iv)-(iv') of Theorem \ref{thm:prod-quadrics}. If $[\Pi^m_{j=1}Q_{q_j})]=[\Pi^{m'}_{j=1}Q_{q'_j})]$ in the Grothendieck ring of varieties $K_0\mathrm{Var}(k)$, then it follows from Theorem \ref{thm:prod-quadrics}(i) that $m=m'$. Moreover, we have $\mu_{\mathrm{JT}}([\Pi^m_{j=1}Q_{q_j})])=\mu_{\mathrm{JT}}([\Pi^m_{j=1}Q_{q'_j})])$ in $R_{\mathrm{B}}(k)$. Thanks to Proposition \ref{prop:auxiliar1}(i), the latter equality holds if and only if we have an isomorphism $U(\perf_\dg(\Pi^m_{j=1}Q_{q_j}))\simeq U(\perf_\dg(\Pi^m_{j=1}Q_{q'_j}))$ in $\NChow(k)$. Hence, the proof of items (iii)-(iv)-(iv') of Theorem \ref{thm:prod-quadrics} follows now from the following result:
\begin{theorem}\label{thm:auxiliar}
Let $\{q_j\}_{1\leq j \leq m}$ and $\{q'_j\}_{1\leq j \leq m}$ be two families of non-degenerate quadratic forms with trivial discriminant of dimension $n\geq 5$. If we have an isomorphism $U(\perf_\dg(\Pi^m_{j=1}Q_{q_j}))\simeq U(\perf_\dg(\Pi^m_{j=1}Q_{q'_j}))$ in the category $\NChow(k)$, then the following holds:
\begin{itemize}
\item[(iii)] When $n=6$ and $m\leq 5$, we have $\Pi_{j=1}^m Q_{q_j}\simeq \Pi_{j=1}^{m}Q_{q'_j}$.
\item[(iv)] When $I^3(k)=0$ and $m \leq 5$, we have $\Pi_{j=1}^m Q_{q_j}\simeq \Pi_{j=1}^{m}Q_{q'_j}$.
\item[(iv')] When $I^3(k)=0$, $m\geq 6$, and the following extra condition holds (consult Notation \ref{not:sums})
\begin{equation}\label{eq:condition-sums1}
\begin{cases} \Sigma^1_{\mathrm{even}}(m,n,l)> \Sigma^2_{\mathrm{even}}(m,n,l)\,\,\,\mathrm{for}\,\,\,\mathrm{all}\,\,\,2 \leq l \leq m-3 & \quad \mathrm{n}\,\,\mathrm{even} \\
\Sigma^1_{\mathrm{odd}}(m,n,l)> \Sigma^2_{\mathrm{odd}}(m,n,l)\,\,\,\mathrm{for}\,\,\,\mathrm{all}\,\,\,2 \leq l \leq m-3 & \quad \mathrm{n}\,\,\mathrm{odd}\,,
\end{cases}
\end{equation}
we also have $\Pi_{j=1}^m Q_{q_j}\simeq \Pi_{j=1}^{m} Q_{q'_j}$.
\end{itemize}
\end{theorem}
\begin{proof}
Recall from the proof of Proposition \ref{prop:cancellation22} that we have the following computation
\begin{equation}\label{eq:computation1}
U(\perf_\dg(Q_{q_j}))\simeq \begin{cases} U(k)^{\oplus (n-2)} \oplus U(C_0^+(q_j))^{\oplus 2} & n\,\,\mathrm{even}\\
U(k)^{\oplus (n-2)} \oplus U(C_0(q_j)) & n\,\,\mathrm{odd}\end{cases}
\end{equation}
in the category $\mathrm{CSA}(k)^\oplus$. Following Remark \ref{rk:generalization}, we have moreover the following isomorphisms:
\begin{equation}\label{eq:computation2}
U(\perf_\dg(\Pi_{j=1}^m Q_{q_j})) \simeq U(\otimes_{j=1}^m \perf_\dg(Q_{q_j})) \simeq \otimes_{j=1}^m U(\perf_\dg(Q_{q_j}))\,.
\end{equation}
By combining \eqref{eq:computation1}-\eqref{eq:computation2}, we hence obtain the following computation
\begin{equation}\label{eq:computation-main}
U(\perf_\dg(\Pi_{j=1}^m Q_{q_j})) \simeq \begin{cases} \oplus_{S\subseteq \{1, \ldots, m\}} U(\otimes_{s \in S}C_0^+(q_s))^{\oplus (2^{\#(S)} \times (n-2)^{m-\#(S)})} & n\,\,\mathrm{even} \\
\oplus_{S\subseteq \{1, \ldots, m\}} U(\otimes_{s \in S}C_0^+(q_s))^{\oplus ((n-2)^{m-\#(S)})} & n\,\,\mathrm{odd}\,, \end{cases}
\end{equation}
where $\#(S)$ stands for the cardinality of $S$ and $\otimes_{s \in \emptyset} C^+_0(q_s) = k$. Recall that $[C_0^+(q_j)], [C_0(q_j)] \in {}_2\mathrm{Br}(k)$. Therefore, if we have an isomorphism $U(\perf_\dg(\Pi^m_{j=1}Q_{q_j}))\simeq U(\perf_\dg(\Pi^m_{j=1}Q_{q'_j}))$ in the category $\NChow(k)$, then it follows from Proposition \ref{prop:key} that the following sets of Brauer classes are the same up to permutation:
\begin{equation}\label{eq:permutation}
\begin{cases} \cup_{S\subseteq \{1, \ldots, m\}}\{\!\!\!\!\underbrace{[\otimes_{s\in S}C_0^+(q_s)]}_{2^{\#(S)}\times (n-2)^{m-\#(S)}}\!\!\!\!\} \quad  \cup_{S\subseteq \{1, \ldots, m\}}\{\!\!\!\!\underbrace{[\otimes_{s\in S}C_0^+(q'_s)]}_{2^{\#(S)}\times (n-2)^{m-\#(S)}}\!\!\!\!\} & \quad n\,\,\mathrm{even} \\
\cup_{S\subseteq \{1, \ldots, m\}}\{\underbrace{[\otimes_{s\in S}C_0(q_s)]}_{(n-2)^{m-\#(S)}}\} \quad  \cup_{S\subseteq \{1, \ldots, m\}}\{\underbrace{[\otimes_{s\in S}C_0(q'_s)]}_{(n-2)^{m-\#(S)}}\} & \quad n\,\,\mathrm{odd}\,. \\
\end{cases}
\end{equation}
Moreover, Proposition \ref{prop:auxiliar1}(ii) yields the following equalities:
\begin{equation}\label{eq:Brauer-gen11}
\begin{cases} \langle \{[C_0^+(q_j)]\}_{1\leq j \leq m}\rangle = \langle \{[C_0^+(q'_j)]\}_{1\leq j \leq m}\rangle & n\,\,\mathrm{even} \\
\langle \{[C_0(q_j)]\}_{1\leq j \leq m}\rangle = \langle \{[C_0(q'_j)]\}_{1\leq j \leq m}\rangle & n\,\,\mathrm{odd}\,.
\end{cases}
\end{equation}
Note that since $[C_0^+(q_j)], [C_0(q_j)] \in {}_2\mathrm{Br}(k)$, \eqref{eq:Brauer-gen11} is a (finite-dimensional) $\bbF_2$-linear subspace of ${}_2\mathrm{Br}(k)$. 
\begin{notation}\label{not:dimension}
Let us write $d$ for the dimension of the $\bbF_2$-vector space \eqref{eq:Brauer-gen11}.
\end{notation}

\smallskip

{\bf Item (iii).} We will prove item (iii) by induction on $m\geq 1$. Note first that in the particular case where $m=1$, the proof follows from Proposition \ref{prop:3-equivalences}. Let us assume that item (iii) holds for $m-1$, with $m\in\{2, 3, 4, 5\}$. Making use of \eqref{eq:computation2}, we have the following isomorphism
\begin{equation}\label{eq:iso-main00}
U(\perf_\dg(\Pi_{j=1}^{m-1}Q_{q_j}))\otimes U(\perf_\dg(Q_{q_m}))\simeq U(\perf_\dg(\Pi_{j=1}^{m-1}Q_{q'_j}))\otimes U(\perf_\dg(Q_{q'_m}))
\end{equation}
in the category $\NChow(k)$. On the one hand, when $m\in \{2, 3, 4\}$, we have $d\in \{0,1,2, m-1,m\}$. On the other hand, $d\in \{0,1,2,3,4,5\}$ when $m=5$. Moreover, when $d=3$, we have the following inequality: 
$$ \Sigma^1_{\mathrm{even}}(5,6,2)=2^9 + 2^7+ 2^7 = 768 > 576 = 2^8 + 2^8 + 2^6 = \Sigma^2_{\mathrm{even}}(5,6,2)\,.$$
Therefore, thanks to Lemma \ref{lem:dimension} below, there exist integers $r$ and $s$ such that $Q_{q_r}\simeq Q_{q'_s}$. Without loss of generality, we can assume that $Q_{q_m}\simeq Q_{q'_m}$. By applying the $\otimes$-cancellation Proposition \ref{prop:cancellation22} to \eqref{eq:iso-main00} (with $Q_q$ equal to $Q_{q_m}\simeq Q_{q'_m}$), we hence obtain an isomorphism $U(\perf_\dg(\Pi_{j=1}^{m-1}Q_{q_j}))\simeq U(\perf_\dg(\Pi_{j=1}^{m-1}Q_{q'_j}))$ in the category $\NChow(k)$. Using the assumption that item (iii) holds for $m-1$, we therefore conclude that $\Pi_{j=1}^{m-1}Q_{q_j} \simeq \Pi_{j=1}^{m-1}Q_{q'_j}$. Consequently, the searched isomorphism $\Pi_{j=1}^{m}Q_{q_j} \simeq \Pi_{j=1}^{m}Q_{q'_j}$ follows now from the combination of $\Pi_{j=1}^{m-1}Q_{q_j} \simeq \Pi_{j=1}^{m-1}Q_{q'_j}$ with $Q_{q_m}\simeq Q_{q'_m}$.

\smallskip

{\bf Item (iv).} The proof is similar to the the proof of item (iii): simply replace the condition $n=6$ by the condition $I^3(k)=0$ and the computations of $\Sigma^1_{\mathrm{even}}(5,6,2)$ and $\Sigma^2_{\mathrm{even}}(5,6,2)$ by the following computations:
$$
\begin{cases}
\Sigma^1_{\mathrm{even}}(5,n,2)=2 \times (n-2)^4 + 2^5 \times (n-2) + 2^3 \times (n-2)^2 \\
\Sigma^1_{\mathrm{odd}}(5,n,2)=  (n-2)^4 + 2\times (n-2) + (n-2)^2\\
\Sigma^2_{\mathrm{even}}(5,n,2)= 2^2 \times (n-2)^3 + 2^2 \times (n-2)^3 + 2^4 \times (n-2) \\
\Sigma^2_{\mathrm{odd}}(5,n,2)= (n-2)^3 + 2\times (n-2)^3 + (n-2)\,.\
\end{cases}
$$
A simple verification shows that $\Sigma^1_{\mathrm{even}}(5,n,2) > \Sigma^2_{\mathrm{even}}(5,n,2)$ and $\Sigma^1_{\mathrm{odd}}(5,n,2) > \Sigma^2_{\mathrm{odd}}(5,n,2)$ when $n\geq 5$.

\smallskip

{\bf Item (iv').} We will prove item (iv') by induction on $m\geq 6$. Using \eqref{eq:computation2}, we have the isomorphism
\begin{equation}\label{eq:iso-main000}
U(\perf_\dg(\Pi_{j=1}^{m-1}Q_{q_j}))\otimes U(\perf_\dg(Q_{q_m}))\simeq U(\perf_\dg(\Pi_{j=1}^{m-1}Q_{q'_j}))\otimes U(\perf_\dg(Q_{q'_m}))
\end{equation}
in the category $\NChow(k)$. Let us assume that item (iv') holds for $m-1$, with $m\geq 6$ (in the case where $m=6$, item (iv) holds). By definition, $d\in \{0,1,2, \ldots, m-1,m\}$. Moreover, when $d\in \{3, \ldots, m-2\}$, the extra condition \eqref{eq:condition-sums1} (with $l=m-d$) implies that $\Sigma^1_{\mathrm{even}}(m,n, m-d)> \Sigma^2_{\mathrm{even}}(m,n, m-d)$ when $n$ is even or $\Sigma^1_{\mathrm{odd}}(m,n, m-d)>\Sigma^2_{\mathrm{odd}}(m,n, m-d)$ when $n$ is odd. Therefore, thanks to Lemma \ref{lem:dimension} below, there exist integers $r$ and $s$ such that $Q_{q_r}\simeq Q_{q'_s}$. Without loss of generality, we can assume that $Q_{q_m}\simeq Q_{q'_m}$. By applying the $\otimes$-cancellation Proposition \ref{prop:cancellation22} to the isomorphism \eqref{eq:iso-main000} (with $Q_q$ equal to $Q_{q_m}\simeq Q_{q'_m}$), we hence obtain an isomorphism $U(\perf_\dg(\Pi_{j=1}^{m-1}Q_{q_j}))\simeq U(\perf_\dg(\Pi_{j=1}^{m-1}Q_{q'_j}))$ in the category $\NChow(k)$. Now, note that Lemma \ref{lem:implications} implies that if the extra condition \eqref{eq:condition-sums1} holds for $m$, then it also holds for $m-1$. Using the assumption that item (iv') holds for $m-1$ (in the case where $m=6$, item (iv) holds), we therefore conclude that $\Pi_{j=1}^{m-1}Q_{q_j} \simeq \Pi_{j=1}^{m-1}Q_{q'_j}$. Consequently, the searched isomorphism $\Pi_{j=1}^{m}Q_{q_j} \simeq \Pi_{j=1}^{m}Q_{q'_j}$ follows now from the combination of $\Pi_{j=1}^{m-1}Q_{q_j} \simeq \Pi_{j=1}^{m-1}Q_{q'_j}$ with $Q_{q_m}\simeq Q_{q'_m}$.
\end{proof}
\begin{lemma}\label{lem:dimension}
Assume that $n=6$ or that $n\geq 5$ and $I^3(k)=0$.
\begin{itemize}
\item[(i)] If $d\in \{0, 1, 2, m-1, m\}$ (consult Notation \ref{not:dimension}), then there exist integers $r$ and $s$ such that $Q_{q_r}\simeq Q_{q'_s}$.
\item[(i')] If $d\in \{3, \ldots, m-2\}$ and the following extra condition holds (consult Notation \ref{not:sums})
\begin{equation}\label{eq:extra}
\begin{cases} \Sigma^1_{\mathrm{even}}(m,n,m-d)> \Sigma^2_{\mathrm{even}}(m,n,m-d) & \mathrm{n}\,\,\mathrm{even} \\
\Sigma^1_{\mathrm{odd}}(m,n,m-d)> \Sigma^2_{\mathrm{odd}}(m,n,m-d) & \mathrm{n}\,\,\mathrm{odd}\,,
\end{cases}
\end{equation}
then there also exist integers $r$ and $s$ such that $Q_{q_r}\simeq Q_{q'_s}$. 
\end{itemize}
\end{lemma}
\begin{proof}
In the case where $d\in \{0, 1, 2\}$, it follows from the equalities \eqref{eq:Brauer-gen11} that there exist integers $r$ and $s$ such that $[C_0^+(q_r)]=[C_0^+(q'_s)]$ when $n$ is even or $[C_0(q_r)]=[C_0(q'_s)]$ when $n$ is odd. Making use of Proposition \ref{prop:3-equivalences}, we hence conclude that $Q_{q_r}\simeq Q_{q'_s}$. In the case where $d=m$, it follows from the precise number of copies of each Brauer class in \eqref{eq:permutation} that there exists a permutation $\sigma$ of the set $\{1, \ldots, m\}$ such that $[C_0^+(q'_j)]=[C_0^+(q_{\sigma(j)})]$ when $n$ is even or $[C_0(q'_j)]=[C_0(q_{\sigma(j)})]$ when $n$ is odd. In particular, there exist integers $r$ and $s$ such that $[C_0^+(q_r)]=[C_0^+(q'_s)]$ when $n$ is even or $[C_0(q_r)]=[C_0(q'_s)]$ when $n$ is odd. Making use of Proposition \ref{prop:3-equivalences}, we hence conclude that $Q_{q_r}\simeq Q_{q'_s}$. In the case where $d=m-1$, let us suppose by absurd that $Q_{q_r}\not\simeq Q_{q'_s}$ for every $1\leq r, s\leq m$. Thanks to Proposition \ref{prop:3-equivalences}, this is equivalent to the condition that $[C_0^+(q_r)]\neq [C_0^+(q'_s)]$ for every $1\leq r, s \leq m$ when $n$ is even or $[C_0(q_r)]\neq[C_0(q'_s)]$ for every $1\leq r, s \leq m$ when $n$ is odd. Without loss of generality, we can assume that the $m-1$ vectors $\{[C_0^+(q_j)]\}_{1\leq j \leq m-1}$, resp. $\{[C_0^+(q'_j)]\}_{1\leq j \leq m-1}$, of the $\bbF_2$-vector space $\langle \{[C_0^+(q_j)]\}_{1\leq j \leq m}\rangle$, resp. $\langle \{[C_0^+(q'_j)]\}_{1\leq j \leq m}\rangle$, are linearly independent when $n$ is even or that the $m-1$ vectors $\{[C_0(q_j)]\}_{1\leq j \leq m-1}$, $\{[C_0(q'_j)]\}_{1\leq j \leq m-1}$, of the $\bbF_2$-vector space $\langle \{[C_0(q_j)]\}_{1\leq j \leq m}\rangle$, resp. $\langle \{[C_0(q'_j)]\}_{1\leq j \leq m}\rangle$, are linearly independent when $n$ is odd. On the one hand, this assumption implies that the {\em lowest} possible number of copies of the Brauer class $[C_0^+(q_1)]$, resp. $[C_0(q_1)]$, on the {\em left-hand side} of \eqref{eq:permutation} is attained when $[C_0^+(q_m)]=[C_0^+(q_2) \otimes \ldots \otimes C_0^+(q_{m-1})]$, resp. $[C_0(q_m)]=[C_0(q_2) \otimes \ldots \otimes C_0(q_{m-1})]$. A simple computation shows that the corresponding number of copies is equal to $\Sigma^1_{\mathrm{even}}(m,n,1)$ when $n$ is even or $\Sigma^1_{\mathrm{odd}}(m,n,1)$ when $n$ is odd. On the other hand, the above assumption implies that the {\em highest} possible number of copies of the Brauer class $[C_0^+(q_1)]$, resp. $[C_0(q_1)]$, on the {\em right-hand side} of \eqref{eq:permutation} is obtained when $[C_0^+(q_1)]=[C_0^+(q'_t) \otimes C_0^+(q'_u)]$, resp. $[C_0(q_1)]=[C_0(q'_t) \otimes C_0(q'_u)]$, for integers $1\leq t\neq u \leq m-1$ and when $[C_0^+(q'_m)]=[C_0^+(q'_v)]$, resp. $[C_0(q'_m)]=[C_0(q'_v)]$, for an integer $1\leq v \leq m-1$ (recall that we are supposing by absurd that $[C_0^+(q_r)]\neq[C_0^+(q'_s)]$ for every $1 \leq r,s \leq m$ when $n$ is even or $[C_0(q_r)]\neq[C_0(q'_s)]$ for every $1 \leq r,s \leq m$ when $n$ is odd). A simple computation shows that the corresponding number of copies is equal to $\Sigma^1_{\mathrm{even}}(m,n,1)$ when $n$ is even or $\Sigma^2_{\mathrm{odd}}(m,n,1)$ when $n$ is odd. Now, note that the following inequalities
$$
\Sigma^1_{\mathrm{even}}(m,n,1)= 2\times (n-2)^{m-1} + 2^m  > 4 \times (n-2)^{m-2} + 2\times (n-2)^{m-2} = \Sigma^2_{\mathrm{even}}(m,n,1) $$
$$
\Sigma^1_{\mathrm{odd}}(m,n,1)= (n-2)^{m-1} + 1 >  (n-2)^{m-2} + (n-2)^{m-2} = \Sigma^2_{\mathrm{odd}}(m,n,1)
$$
lead to a contradiction with the fact that the above two sets \eqref{eq:permutation} are the same up to permutation. Consequently,  there exist integers $r$ and $s$ such that $[C_0^+(q_r)]=[C_0^+(q'_s)]$ when $n$ is even or $[C_0(q_r)]=[C_0(q'_s)]$ when $n$ is odd. Making use of Proposition \ref{prop:3-equivalences}, we hence conclude that $Q_{q_r}\simeq Q_{q'_s}$. This proves item (i).

We now prove item (i'). Let us suppose by absurd that $Q_{q_r}\not\simeq Q_{q'_s}$ for every $1 \leq r, s \leq m$. Thanks to Proposition \ref{prop:3-equivalences}, this is equivalent to the condition that $[C_0^+(q_r)]\neq [C_0^+(q'_s)]$ for every $1\leq r, s \leq m$ when $n$ is even or $[C_0(q_r)]\neq [C_0(q'_s)]$ for every $1\leq r, s \leq m$ when $n$ is odd. Since $d\in \{3, \ldots, m-2\}$, we can assume, without loss of generality, that the $d$ vectors $\{[C_0^+(q_j)]\}_{1\leq j \leq d}$, resp. $\{[C_0^+(q'_j)]\}_{1\leq j \leq d}$, of the $\bbF_2$-vector space $\langle \{[C_0^+(q_j)]\}_{1\leq j \leq m}\rangle$, resp. $\langle \{[C_0^+(q'_j)]\}_{1\leq j \leq m}\rangle$, are linearly independent when $n$ is even or that the $d$ vectors $\{[C_0(q_j)]\}_{1\leq j \leq d}$, resp. $\{[C_0(q'_j)]\}_{1\leq j \leq d}$, of the $\bbF_2$-vector space $\langle \{[C_0(q_j)]\}_{1\leq j \leq m}\rangle$, resp. $\langle \{[C_0(q'_j)]\}_{1\leq j \leq m}\rangle$, are linearly independent when $n$ is odd. On the one hand, this assumption implies that the {\em lowest} possible number of copies of the Brauer class $[C_0^+(q_1)]$, resp. $[C_0(q_1)]$, on the {\em left-hand side} of \eqref{eq:permutation} is attained when $[C_0^+(q_{d+1})]=\cdots = [C_0^+(q_m)]=[C_0^+(q_2)\otimes \cdots \otimes C_0^+(q_{m-1})]$, resp. $[C_0(q_{d+1})]=\cdots = [C_0(q_m)]=[C_0(q_2)\otimes \cdots \otimes C_0(q_{m-1})]$. A simple computation shows that the corresponding number of copies is equal to $\Sigma^1_{\mathrm{even}}(m, n, m-d)$ when $n$ is even or $\Sigma^1_{\mathrm{odd}}(m,n,m-d)$ when $n$ is odd. On the other hand, the above assumption implies that the {\em highest} possible number of copies of Brauer class $[C_0^+(q_1)]$, resp. $[C_0(q_1)]$, on the {\em right-hand side} of \eqref{eq:permutation} is attained when $[C_0^+(q_1)]=[C_0^+(q'_t) \otimes C_0^+(q'_u)]$, resp. $[C_0(q_1)]=[C_0(q'_t) \otimes C_0(q'_u)]$, for integers $1\leq t \neq u \leq d$ and when $[C_0^+(q'_{d+1})]=\cdots =[C_0^+(q'_m)]=[C_0^+(q'_v)]$, resp. $[C_0(q'_{d+1})]=\cdots =[C_0(q'_m)]=[C_0(q'_v)]$, for an integer $1\leq v \leq d$. A simple computation shows that the corresponding number of copies is equal to $\Sigma^2_{\mathrm{even}}(m,n,m-d)$ when $n$ is even or $\Sigma^2_{\mathrm{odd}}(m,n,m-d)$ when $n$ is odd. Thanks to the inequalities \eqref{eq:extra}, we hence obtain a contradiction with the fact that the above two sets \eqref{eq:permutation} are the same up to permutation. Consequently, there exist integers $r$ and $s$ such that $[C_0^+(q_r)]=[C_0^+(q'_s)]$ when $n$ is even or $[C_0(q_r)]=[C_0(q'_s)]$ when $n$ is odd. Making use of Proposition \ref{prop:3-equivalences}, we hence conclude that $Q_{q_r}\simeq Q_{q'_s}$. 
\end{proof}
\section{Proof of Theorem \ref{thm:app5}}
{\bf Item (i).} Recall first that $\mathrm{dim}(\mathrm{Iv}(A,\ast))=\mathrm{deg}(A)$. Following \S\ref{sub:involution}, note that $\mu_\rho(\mathrm{Iv}(A,\ast))=\mathrm{deg}(A)$. Hence, if $[\mathrm{Iv}(A,\ast)]=[\mathrm{Iv}(A',\ast')]$ in $K_0\mathrm{Var}(k)$, we conclude that $\mathrm{dim}(\mathrm{Iv}(A,\ast))=\mathrm{dim}(\mathrm{Iv}(A',\ast'))$. This is equivalent to the equality $\mathrm{deg}(A)=\mathrm{deg}(A')$. 

\smallskip

{\bf Item (ii).} If $[\mathrm{Iv}(A,\ast)]=[\mathrm{Iv}(A',\ast')]$ in $K_0\mathrm{Var}(k)$, then it follows from Proposition \ref{prop:auxiliar1}(ii) that
\begin{equation}\label{eq:equality1}
\langle [A], [C_0^+(A,\ast)], [C_0^-(A,\ast)]\rangle = \langle [A'], [C_0^+(A',\ast')], [C_0^-(A',\ast')] \rangle \,.
\end{equation}
When $\mathrm{deg}(A)\equiv 2$ (mod $4$), we have the following relations in the Brauer group:
\begin{eqnarray}\label{eq:relations1}
2[C_0^+(A,\ast)]=[A] & 3[C_0^+(A,\ast)]=[C_0^-(A,\ast)] & 4[C_0^+(A,\ast)]=[k]\,.
\end{eqnarray}
Note that \eqref{eq:equality1}-\eqref{eq:relations1} imply that $[C_0^+(A,\ast)]=[C_0^+(A',\ast')]$ and $[C_0^-(A,\ast)]=[C_0^-(A',\ast')]$ or that $[C_0^+(A,\ast)]=[C_0^-(A',\ast')]$ and $[C_0^-(A,\ast)]=[C_0^+(A',\ast')]$. Using the fact that $\mathrm{deg}(C_0^+(A,\ast))=\mathrm{deg}(C_0^-(A,\ast))=2^{\frac{\mathrm{deg}(A)}{2}-1}$ and that $\mathrm{deg}(A)=\mathrm{deg}(A')$ (proved in item (i)), we hence conclude that $C_0^\pm(A,\ast)\simeq C_0^\pm(A',\ast')$. 

When $\mathrm{deg}(A)\equiv 0$ (mod $4$), we have the following relations in the Brauer group:
\begin{eqnarray}\label{eq:relations}
2[C_0^+(A,\ast)]=[k] & 2[C_0^-(A,\ast)]=[k] & [C_0^+(A,\ast)]+[C_0^-(A,\ast)]=[A]\,.
\end{eqnarray}
In this case, we need to consider also the noncommutative Chow motive of the involution variety $\mathrm{Iv}(A,\ast)$. Recall from \cite[Example~3.11]{Homogeneous} that we have the following computation
\begin{equation}\label{eq:computation}
U(\perf_\dg(\mathrm{Iv}(A,\ast)))\simeq U(k)^{\oplus \frac{\mathrm{deg}(A)}{2}-1} \oplus U(A)^{\oplus \frac{\mathrm{deg}(A)}{2}-1} \oplus U(C^+_0(A,\ast)) \oplus U(C_0^-(A,\ast))
\end{equation}
in the category $\NChow(k)$. If $[\mathrm{Iv}(A,\ast)]=[\mathrm{Iv}(A',\ast')]$ in the Grothendieck ring of varieties $K_0\mathrm{Var}(k)$, then $\mu_{\mathrm{JT}}([\mathrm{Iv}(A,\ast)])=\mu_{\mathrm{JT}}([\mathrm{Iv}(A',\ast')])$ in $R_{\mathrm{B}}(k)$. Thanks to Proposition \ref{prop:auxiliar1}(i), the latter equality holds if and only if we have an isomorphism $U(\perf_\dg(\mathrm{Iv}(A,\ast)))\simeq U(\perf_\dg(\mathrm{Iv}(A',\ast')))$ in the category $\NChow(k)$. Note that the relations \eqref{eq:relations} imply, in particular, that $[A], [C_0^+(A,\ast)], [C_0^-(A,\ast)] \in {}_2\mathrm{Br}(k)$. Therefore, making use of Proposition \ref{prop:key}, we conclude that the noncommutative Chow motives $U(\perf_\dg(\mathrm{Iv}(A,\ast)))$ and $U(\perf_\dg(\mathrm{Iv}(A',\ast')))$ are isomorphic if and only if the following two sets of Brauer classes 
\begin{eqnarray*}\label{eq:sets}
\{\!\!\!\!\!\underbrace{[k]}_{\frac{\mathrm{deg}(A)}{2}-1}, \underbrace{[A]}_{\frac{\mathrm{deg}(A)}{2}-1}, \underbrace{[C_0^+(A,\ast)]}_{1}, \underbrace{[C_0^-(A,\ast)]}_{1} \} && \{\!\!\!\!\!\underbrace{[k]}_{\frac{\mathrm{deg}(A')}{2}-1}, \underbrace{[A']}_{\frac{\mathrm{deg}(A')}{2}-1}, \underbrace{[C_0^+(A',\ast')]}_{1}, \underbrace{[C_0^-(A',\ast')]}_{1} \}
\end{eqnarray*}
are the same up to permutation, where the numbers below the parenthesis denote the number of copies. Since $\mathrm{deg}(A)\geq 6$, there are at least $2$ copies of $[k]$ and $[A]$. This fact, combined with the relations \eqref{eq:relations}, implies that $[C_0^+(A,\ast)]=[C_0^+(A',\ast')]$ and $[C_0^-(A,\ast)]=[C_0^-(A',\ast')]$ or that $[C_0^+(A,\ast)]=[C_0^-(A',\ast')]$ and $[C_0^-(A,\ast)]=[C_0^+(A',\ast')]$. As above, using the equalities $\mathrm{deg}(C_0^+(A,\ast))=\mathrm{deg}(C_0^-(A,\ast))=2^{\frac{\mathrm{deg}(A)}{2}-1}$ and $\mathrm{deg}(A)=\mathrm{deg}(A')$, we hence conclude that $C_0^\pm(A,\ast)\simeq C_0^\pm(A',\ast')$.

\smallskip
 
{\bf Item (iii).} When $\mathrm{deg}(A)=6$, the assignment $(A,\ast) \mapsto C_0^+(A,\ast) \times C_0^-(A,\ast)$ gives rise to a one-to-one correspondence between isomorphism classes of central simple $k$-algebras of degree $6$ with involution of orthogonal type and trivial discriminant and isomorphism classes of $k$-algebras of the form $Q\times Q^\op$, where $Q$ is a quaternion algebra; consult \cite[Cor.~15.32]{Book-inv}. Note that the $k$-algebras $C_0^+(A,\ast) \times C_0^-(A,\ast)$ and $C_0^+(A',\ast') \times C_0^-(A',\ast')$ are isomorphic if and only if $C_0^\pm(A,\ast)\simeq C_0^\pm(A',\ast')$. Consequently, the proof follows from item (ii) and from the general fact that two central simple $k$-algebras with involution of orthogonal type $(A,\ast)$ and $(A',\ast')$ are isomorphic if and only if the involution varieties $\mathrm{Iv}(A,\ast)$ and $\mathrm{Iv}(A',\ast')$ are isomorphic.

\smallskip

{\bf Item (iv).} When $I^3(k)=0$, we have the following classification theorem: if $\mathrm{deg}(A)=\mathrm{deg}(A')$ and $C_0^\pm(A,\ast)\simeq C_0^\pm(A',\ast')$, then the central simple algebras with involution of orthogonal type $(A,\ast)$ and $(A',\ast')$ are isomorphic; consult \cite[Thm.~A]{LewisT}. Consequently, the proof follows from the combination of items (i)-(ii) with the general fact that two central simple $k$-algebras with involution of orthogonal type $(A,\ast)$ and $(A',\ast')$ are isomorphic if and only if the involution varieties $\mathrm{Iv}(A,\ast)$ and $\mathrm{Iv}(A',\ast')$ are isomorphic.

\smallskip

\medbreak\noindent\textbf{Acknowledgments.} I am grateful to Michael Artin for enlightning discussions about Severi-Brauer varieties, to Marcello Bernardara for a stimulating discussion about the Amitsur's conjecture, and to Asher Auel for the references \cite{ELam,LewisT}. I am also very grateful to the Institut des Hautes \'Etudes Scientifiques (IH\'ES) and to the Max-Planck-Institut f\"ur Mathematik (MPIM) for their hospitality, where this work was finalized.

\end{document}

\end{proof}